\documentclass[11pt]{amsart}

\usepackage{epsf,amssymb,times,overpic,amscd}
\usepackage[usenames]{color}
\usepackage[all]{xy}

\usepackage{amsfonts}
\usepackage{amsmath}
\usepackage{graphicx}
\usepackage{epsfig}
\usepackage{epstopdf}

\newcommand{\ba}{\begin{array}}
\newcommand{\ea}{\end{array}}

\newtheorem{thm}{Theorem}[section]
\newtheorem{prop}[thm]{Proposition}
\newtheorem{lemma}[thm]{Lemma}
\newtheorem{cor}[thm]{Corollary}

\theoremstyle{definition}
\newtheorem{defn}[thm]{Definition}

\theoremstyle{remark}

\newtheorem{rmk}[thm]{Remark}
\newtheorem{example}[thm]{Example}

\newcommand{\C}{\mathcal{C}}

\newcommand{\Z}{\mathbb{Z}}
\newcommand{\Q}{\mathbb{Q}}

\newcommand{\bdry}{\partial}

\newcommand{\n}{\noindent}
\newcommand{\F}{\mathbb{F}_2}
\newcommand{\zq}{\Z[q^{\pm1}]}
\newcommand{\uq}{\mathbf{U}_q(\mathfrak{sl}(1|1))}
\newcommand{\rn}{R_n}
\newcommand{\rnn}{R_n \otimes R_n}
\newcommand{\drnn}{R_n \boxtimes R_n}
\newcommand{\drr}{R \boxtimes R}
\newcommand{\rr}{R \otimes R}
\newcommand{\ra}{\rightarrow}
\newcommand{\xra}{\xrightarrow}

\newcommand{\g}{\Gamma}
\newcommand{\gn}{\Gamma_n}
\newcommand{\gnn}{\Gamma_n \boxtimes \Gamma_n}

\newcommand{\mf}{\mathbf}
\newcommand{\op}{\operatorname}
\newcommand{\cal}{\mathcal}
\newcommand{\es}{\emptyset}
\newcommand{\lan}{\langle}
\newcommand{\ran}{\rangle}
\newcommand{\ot}{\otimes}

\newcommand{\be}{\begin{enumerate}}
\newcommand{\ee}{\end{enumerate}}

\begin{document}
\title{A categorification of $\uq$ as an algebra}

\author{Yin Tian}
\address{University of Southern California, Los Angeles, CA 90089}
\email{yintian@usc.edu}


\begin{abstract}
We construct families of differential graded algebras $\rn$ and $\drnn$ for $n>0$ and give an algebraic formulation of the contact category of a disk through the differential graded category $DGP(\rn)$ consisting of finitely iterated mapping cones of maps between some distinguished projective differential graded $\rn$-modules. The $0$-th homology category $H^0(DGP(\rn))$ of $DGP(\rn)$ is a triangulated category and its Grothendieck group $K_0(\rn)$ is a Clifford algebra. We then categorify the multiplication on $K_0(\rn)$ to a functor $DGP(\drnn) \ra DGP(\rn)$. We construct a subcategory of $H^0(DGP(\rn))$ which categorifies an integral version of $\uq$ as an algebra.  	
\end{abstract}
\maketitle

\section{Introduction}
Categorification is a process in which we lift an integer to a vector space, a vector space to a category, and a linear map between vector spaces to a functor between categories.
Two of the poineering works are Khovanov homology defined by Khovanov \cite{Kh1} and knot Floer homology, defined independently by Ozsv\'ath-Szab\'o \cite{OS} and Rasmussen \cite{Ra}, which respectively categorify the Jones and Alexander polynomials respectively.
Khovanov homology and knot Floer homology are finer invariants of knots which take values in the homotopy category of chain complexes of graded vector spaces whose graded Euler characteristics agree with the polynomial invariants.

The Jones polynomial fits in the general representation-theoretic framework of Reshetikhin-Turaev invariants \cite{RT} which applies to any quantum group $\mathbf{U}_q(\mathfrak{g})$ of a semisimple Lie algebra $\mathfrak{g}$ and any finite dimensional representation of $\mathbf{U}_q(\mathfrak{g})$.
In particular, the Jones polynomial can be recovered from the fundamental representation $V_1$ of $\mathbf{U}_q(\mathfrak{sl}_2)$.
With an eye towards categorifying the Reshetikhin-Turaev invariants, Bernstein-Frenkel-Khovanov \cite{BFK} formulated a program for categorifying representations of $\mathbf{U}_q(\mathfrak{sl}_2)$.
The symmetric powers $V_1^{\otimes n}$ of $\mathbf{U}(\mathfrak{sl}_2)$ were categorified in \cite{BFK} and extended to the graded case of $\mathbf{U}_q(\mathfrak{sl}_2)$ by Stroppel \cite{Str}.
Other tensor product representations of $\mathbf{U}_q(\mathfrak{sl}_2)$ were categorified by Frenkel-Khovanov-Stroppel \cite{FKS}.
Chuang and Rouquier \cite{CR} categorified locally finite $\mathfrak{sl}_2$-representations.
More generally, Rouquier \cite{Rou} constructed a $2$-category associated with a Kac-Moody algebra and studied its $2$-representation.
For the quantum groups themselves, Lauda \cite{La} gave a diagrammatic categorification of $\mathbf{U}_q(\mathfrak{sl}_2)$ and Khovanov-Lauda \cite{KL1, KL3, KL2} extended it to the cases of $\mathbf{U}_q(\mathfrak{sl}_n)$ and one-half of the quantum groups associated to an arbitrary Cartan datum.
Cautis and Lauda \cite{CLa} showed the relationship between $2$-representations of a Kac-Moody Lie algebra and those of categorified quantum groups.
The program of categorifying Reshetikhin-Turaev invariants was brought to fruition by Webster \cite{Web1, Web2} using this diagrammatic approach.

For the Alexander polynomial, Kauffman-Saleur developed a representation-theoretic approach in the spirit of \cite{KS} via the quantum group $\uq$ of the super Lie algebra $\mathfrak{sl}(1|1)$.
Rozansky-Saleur in \cite{RS} gave an associated quantum field theory description.
It is therefore natural to ask whether there is a categorical program for $\uq$ and its fundamental representation which could recover knot Floer homology.
The first step in such a program is to categorify $\uq$.
Motivated by the strands algebra of Lipshitz-Ozsv\'ath-Thurston \cite{LOT}, Khovanov \cite{Kh2} categorified the positive part of $\mathbf{U}_q (\mathfrak{gl}(1|2))$.
In related work, Douglas-Manolescu \cite{DM} generalized the strands algebra associated to a surface to a differential 2-algebra associated to a circle.
Sartori \cite{Sa} has recently announced a categorification of tensor products of the fundamental representation of $\mathfrak{gl}(1|1)$ using completely different methods.

The goal of this paper is to present a categorification of the algebra structure of an integral version of $\uq$.
\begin{defn}
The quantum group $\uq$ is the unital associative $\Q(q)$-algebra with generators $E, F, H, H^{-1}$ and relations:
\begin{gather*}
HH^{-1}= H^{-1}H =1, \\
E^2=F^2=0, \\
HE=EH, HF=FH, \\
EF+FE=\frac{H - H^{-1}}{q- q^{-1}}.
\end{gather*}
\end{defn}
We consider two variants of $\uq$: the {\em idempotent completion} $\mf{U}$ and the {\em integral form} $\mf{U}_n$ of $\mf{U}$.
The idempotent completion $\mf{U}$ is obtained from $\uq$ by replacing the unit by a collection of orthogonal idempotents $1_n$ for $n \in \Z$ such that
$$1_n1_m=\delta_{n,m}, \hspace{.2cm} H1_n=1_nH=q^n1_n,\hspace{.2cm} 1_nE=E1_n, \hspace{.2cm} 1_nF=F1_n.$$

\begin{defn} \label{U_n}
The integral form $\mf{U}_n$ is the unital associative $\Z[q^{\pm1}]$-algebra with generators $E, F$ and relations:
\begin{equation*}
E^2=F^2=0, \hspace{.3cm}
EF+FE=\frac{q^{n} - q^{-n}}{q- q^{-1}}=q^{n-1}+\cdots+q^{1-n}.
\end{equation*}
\end{defn}

The algebra structure of $\mf{U}_n$ can be written as a $\Z[q^{\pm1}]$-linear map $\op{m}_n: \mf{U}_n \otimes_{\Z[q^{\pm1}]} \mf{U}_n \ra \mf{U}_n$.

Our main theorem is a categorification of $\op{m}_n$ via triangulated categories for $n>0$.

\begin{thm} [Main theorem] \label{main}
There exist triangulated categories $\mathcal{U}_{n,n}$ and $\mathcal{U}_{n}$ whose Grothendieck groups are $\mf{U}_n \otimes \mf{U}_n$ and $\mf{U}_n$ respectively for $n>0$.
There exists an exact functor $\mathcal{F}_n: \mathcal{U}_{n,n} \rightarrow \mathcal{U}_{n}$ whose induced map on the Grothendieck groups
$K_0(\mathcal{F}_n): K_0(\mathcal{U}_{n,n}) \ra K_0(\mathcal{U}_n)$ agrees with the multiplication map $\op{m}_n: \mf{U}_n \otimes \mf{U}_n \ra \mf{U}_n$.
\end{thm}

The motivation is from the contact category introduced by Honda \cite{Honda1}, which presents an algebraic way to study contact topology in dimension 3.
The contact category $\C(\Sigma)$ of $\Sigma$ is an additive category associated to a compact surface $\Sigma$.
The objects of $\C(\Sigma)$ are isotopy classes of {\em dividing sets} on $\Sigma$ with some homotopy grading.
The morphisms are generated by {\em tight} contact structures on $\Sigma \times [0,1]$ with prescribed dividing sets on $\Sigma \times \{0,1\}$.
More precisely, a dividing set on $\Sigma$ is a properly embedded 1-manifold, possibly disconnected and possibly with boundary, which divides $\Sigma$ into positive and negative regions.
Any dividing set with a contractible component is defined as the zero object since there is no tight contact structure in a neighborhood of the dividing curve by a criterion of Giroux \cite{Gi}.
As basic blocks of morphisms, {\em bypass attachments} introduced by Honda \cite{Honda2} locally change dividing sets as shown in Fig \ref{triangle}.

\begin{figure}[h]
\begin{overpic}
[scale=0.3]{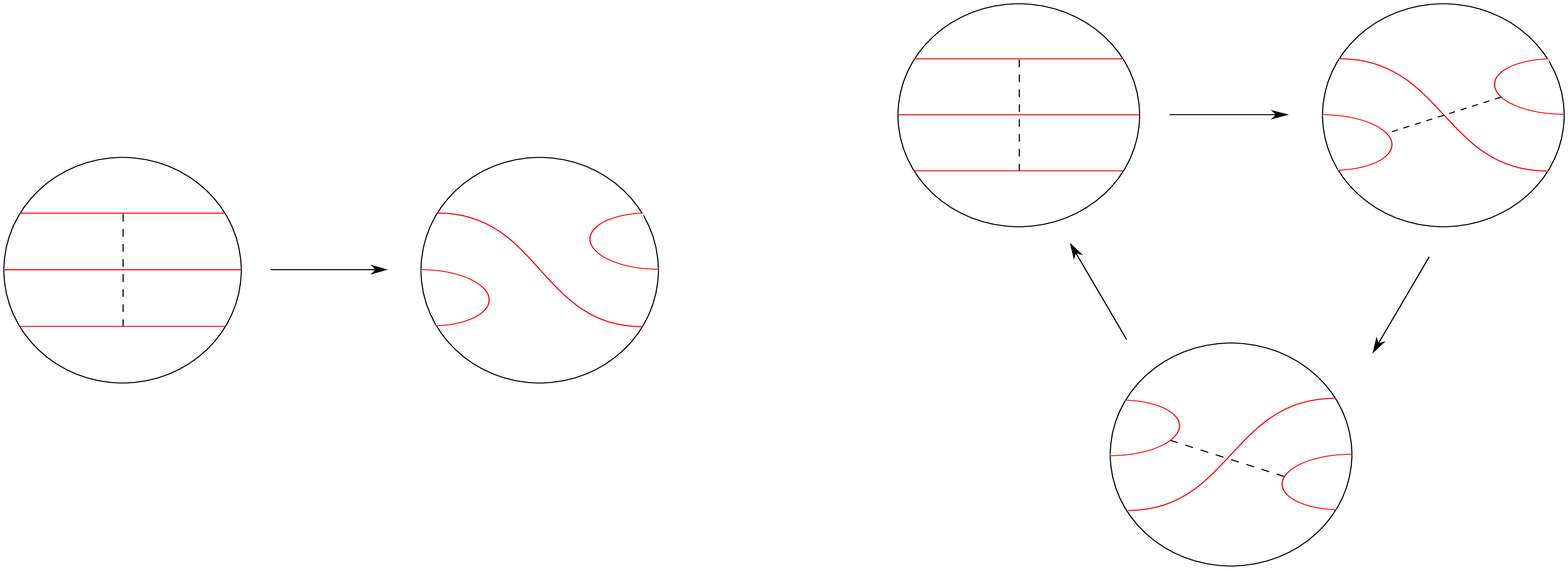}
\end{overpic}
\caption{The left is a bypass attachment; the right is a distinguished triangle given by a triple of bypass attachments.}
\label{triangle}
\end{figure}
The main feature of the contact category $\C(\Sigma)$ is the existence of distinguished triangles given by a triple of bypass attachments as shown in Fig \ref{triangle}.
The homotopy grading associated to dividing sets is related to the shift functor in a triangulated category.
In particular, Huang \cite{Huang} showed that a triple of bypass attachments changes the homotopy grading by $1$.
The triangulated structure was also studied by Mathews \cite{Ma} from a combinatorial perspective.

This paper can be viewed as an algebraic formulation of the triangulated structure on the contact category of a disk, viewed as a rectangle.
For $n>0$, let $\C_n$ be the contact category of a rectangle $D_n$ with $n+2$ marked points on both the left and right sides of $\bdry D_n$ and no marked points on the top and bottom sides.
The marked points are intersection points of dividing sets with $\bdry D_n$.
Then $\C_n$ is a monoidal category equipped with a bifunctor $\rho_n: \C_n \times \C_n \ra \C_n$.
The monoidal structure $\rho_n$ is given by horizontally stacking two dividing sets along their common boundaries for the objects and sideways stacking two contact structures for the morphisms as shown in Fig \ref{obj}.
\begin{figure}[h]
\begin{overpic}
[scale=0.3]{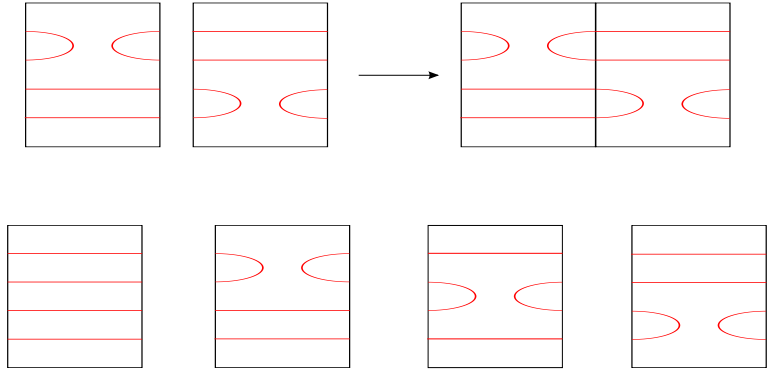}
\end{overpic}
\caption{The top is a stacking of two dividing sets; the bottom is a distinguished collection of dividing sets: $[\es], [0], [1], [2]$.}
\label{obj}
\end{figure}

There is a collection of distinguished dividing sets $[\es]$ and $[i]$'s for $0 \leq i \leq n$ as shown in Fig \ref{obj}.
Note that any nonzero dividing set can be represented as a horizontal stacking of a sequence in the collection up to isotopy.
Let $[i]\cdot[j]$ denote the horizontal stacking of dividing sets $[i]$ and $[j]$.
Let $X_{\es}$ and $X_i$'s be classes of $[\es]$ and $[i]$'s in the Grothendieck group of $\C_n$.
They satisfy the following properties illustrated in Fig \ref{relation}:
\begin{figure}[h]
\begin{overpic}
[scale=0.3]{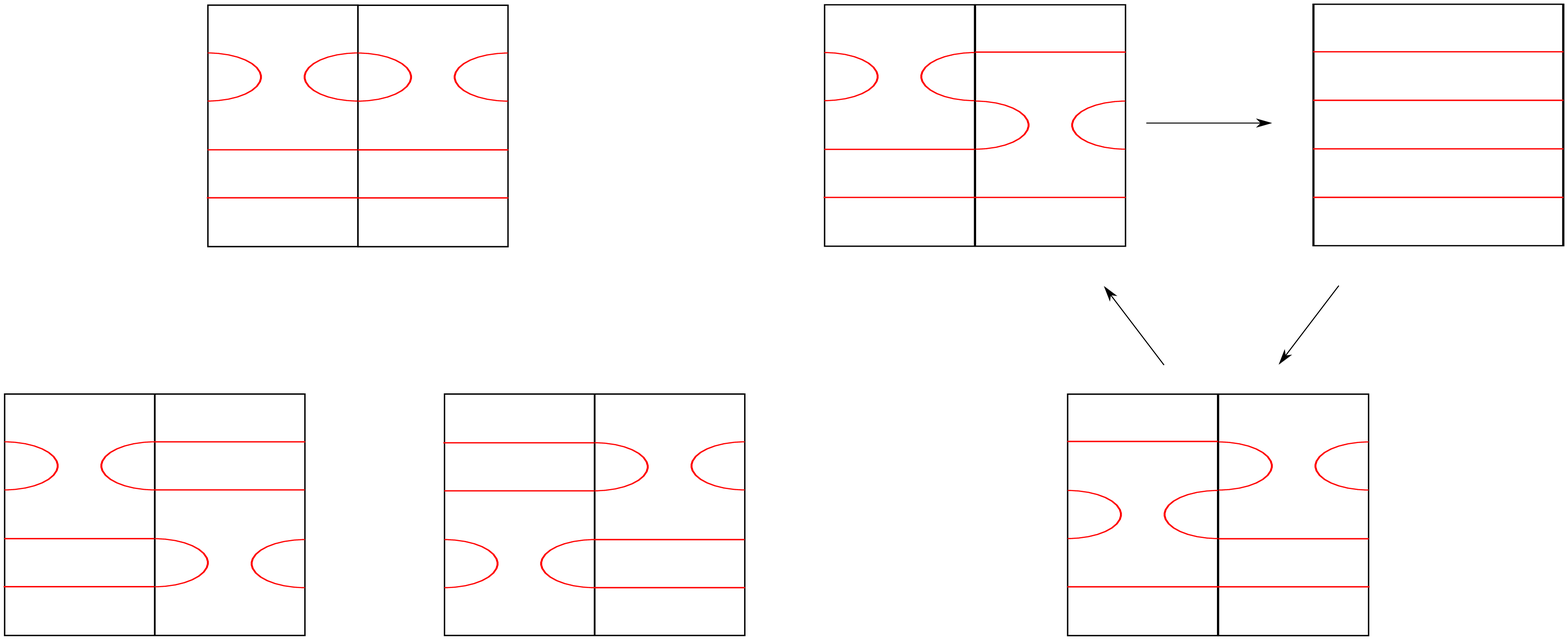}
\end{overpic}
\caption{The top left picture represents $[0]\cdot[0]$; the bottom left compares $[0]\cdot[2]$ and $[2]\cdot[0]$; the right picture is a distinguished triangle: $[0]\cdot[1]\ra[\es]\ra[1]\cdot[0]$.}
\label{relation}
\end{figure}
\be
\item $X_{\es}$ is the unit since any dividing set is unchanged when stacking $[\es]$ from both left and right.
\item $X_i^2=0$ since the dividing set $[i]\cdot[i]$ contains a contractible loop.
\item $X_iX_j=-X_jX_i$ for $|i-j|>1$, since dividing sets $[i]\cdot[j]$ and $[j]\cdot[i]$ are in the same isotopy class as dividing sets, but their homotopy gradings differ by $1$. Hence their classes differ by a minus sign in the Grothendieck group.
\item $X_iX_{i+1}+X_{i+1}X_i=q^{2i+1-n}X_{\es}$, since there exists a distinguished triangle: $[i]\cdot[i+1]\ra[\es]\ra[i+1]\cdot[i]$. The exponent of $q$ is related to the height of the location of the distinguished triangle.
\ee

Given the distinguished objects $[\es]$ and $[i]$'s, we construct a quiver $\gn$ as follows.
The set of vertices $V(\gn)$ consists of objects which are obtained by horizontally stacking the distinguished ones in decreasing order, i.e., $[i_0] \cdot [i_1] \cdots [i_k]$ where $n\geq i_0 > \cdots > i_k \geq 0$.
There exists an arrow $[i_0] \cdot [i_1] \cdots [i_k] \ra [j_0] \cdot [j_1] \cdots [j_l]$ if $l = k + 2$ and $\{j_0,j_1,\cdots,j_l\}=\{i_0,i_1,\cdots,i_k\} \sqcup \{s+1, s\}$ for some $s$.
Each arrow actually represents a tight contact structure given by a single bypass attachment.
For example there is an arrow $[\es]\ra[1]\cdot[0]$ in the distinguished triangle $[0]\cdot[1]\ra[\es]\ra[1]\cdot[0]$.
We form a quotient of the path algebra $\F\gn$ of the quiver $\gn$, where $\F$ is the field of 2 elements.
The relations come from the fact that two disjoint bypass attachments commute up to isotopy.
The graded algebra $\rn$ is then obtained by adding an extra $q$-grading to the quotient.
We view $\rn$ as a $q$-graded DG algebra $(\rn, d=0)$ which has trivial differential and is concentrated in cohomological grading $0$.

At this point we pass to algebra\footnote{In fact, the rest of the paper is just algebra which is motivated by the contact category.} and model $\C_n$ by $DGP(\rn)$, a full subcategory of a DG category $DG(\rn)$ of $q$-graded DG $\rn$-modules.
More precisely, $DGP(\rn)$ consists of finitely iterated mapping cones of maps between some distinguished projective DG $\rn$-modules.
As triangulated categories, the $0$-th homology category $H^0(DGP(\rn))$ is equivalent to $\mf{K}^{b}(\rn)$,
the homotopy category of bounded cochain complex of finitely generated projective $q$-graded $\rn$-modules.
The Grothendieck group $K_0(\rn)$ of $H^0(DGP(\rn))$, is a free $\Z[q^{\pm1}]$-module over the vertex set $V(\gn)$.

Motivated by the horizontal stacking of dividing sets on $\C_n$, we define a multiplication on $K_0(\rn)$, i.e., $\op{m}_n : K_0(\rn) \otimes K_0(\rn) \ra K_0(\rn)$.
To categorify $K_0(\rn) \otimes K_0(\rn)$, we construct two DG algebras: $\drnn$ and its cohomology algebra $H(\drnn)=\rnn$.
Similarly, consider a DG category $DGP(\drnn)$ as a full subcategory of $DG(\drnn)$ such that its $0$-th homology category $H^0(DGP(\drnn))$ is equivalent to $\mf{K}^{b}(\rnn)$.
The Grothendieck group $K_0(H^0(DGP(\drnn)))$ of $H^0(DGP(\drnn))$ is isomorphic to $K_0(\rn) \otimes K_0(\rn)$.
We then construct a functor $\cal{M}_n: DGP(\drnn) \ra DGP(\rn)$ by tensoring with a DG $(\rn, \drnn)$-bimodule $T_n$.
We show that $\cal{M}_n$ induces an exact functor $\cal{M}_n|_{H^0}: H^0(DGP(\drnn)) \ra H^0(DGP(\rn))$ and $\cal{M}_n|_{H^0}$ categorifies the multiplication map $\op{m}_n$ on the level of Grothendieck group.
$$
\xymatrix{
  DGP(\drnn) \ar[r]^{\cal{M}_n} \ar[d]_{H^{0}} & DGP(\rn) \ar[d]_{H^{0}} \\
 H^0(DGP(\drnn)) \ar[r]^{\cal{M}_n|_{H^0}} \ar[d]_{K_{0}} & H^0(DGP(\rn)) \ar[d]_{K_{0}} \\
 K_0(\rnn) \ar[r]^{\op{m}_n}  & K_0(\rn).
}$$
Let $P([\es])$ and $P([i])$ be projective $\rn$-modules corresponding to vertices $[\es]$ and $[i]$ in $V(\gn)$ and let $X_{\es}$ and $X_i$ be the classes of $P([\es])$ and $P([i])$ in $K_0(\rn)$.
Then we have the following:
\begin{thm} \label{Clifford}
The multiplication $\op{m}_n$ makes $K_0(\rn)$ into an associative $\Z[q^{\pm1}]$-algebra with unit $X_{\es}$, generators $X_i$'s for $0\leq i \leq n$, and relations:
\begin{gather*}
X_i^2=0;\\
X_iX_j=-X_jX_i \hspace{.2cm} \mbox{if} \hspace{.2cm} |i-j|>1;\\
X_iX_{i+1}+X_{i+1}X_i=q^{2i+1-n} .
\end{gather*}
There exists the exact functor $\cal{M}_n|_{H^0}: H^0(DGP(\drnn)) \ra H^0(DGP(\rn))$ which categorifies the multiplication $\op{m}_n$.
\end{thm}

The algebra $K_0(\rn)$ is a {\em Clifford algebra} $Cl(V_n, Q_n)$ over $\zq$, where $V_n$ is a free $\zq$-module generated by $\{X_i ~|~ 0 \leq i \leq n\}$ and $Q_n$ is a quadratic form on $V_n$ given by:
$$Q_n(\sum \limits_{i=0}^{n}a_i X_i)=\sum \limits_{i=0}^{n-1}a_{i}a_{i+1} q^{2i+1-n},$$
for $a_i \in \zq$.
Note that the basis $\{X_i ~|~ 0 \leq i \leq n\}$ is different from the basis $\{\psi_i, \psi_i^* ~|~ i \in \Z\}$ of the Clifford algebra in \cite{Kac}.
In the categorification of a {\em Heisenberg algebra}, Khovanov \cite{Kh3} used another basis which is different from the basis $\{p_i, q_i ~|~ i \in \Z\}$.

Finally, we view $\mf{U}_n$ as a subalgebra of $K_0(\rn)$ by setting
$$E = \sum \limits_{\tiny{\begin{array}{c}0\leq i \leq n\\i \hspace{.1cm} \mbox{even}\end{array}}} X_i, \hspace{1cm}
F = \sum \limits_{\tiny{\begin{array}{c}0\leq i \leq n\\i \hspace{.1cm} \mbox{odd}\end{array}}} X_i
$$
It is easy to verify that $E$ and $F$ satisfy the defining relations of $\mf{U}_n$.
We then formally construct triangulated full subcategories $\cal{U}_n$ of $H^0(DGP(\rn))$ and $\cal{U}_{n, n}$ of $H^0(DGP(\drnn))$. The Grothendieck groups of $\cal{U}_n$ and $\cal{U}_{n, n}$ are $\mf{U}_n$ and $\mf{U}_n\otimes \mf{U}_n$. Then the restriction $\mathcal{F}_n$ of $\mathcal{M}_n|_{H^0}$ to $\mathcal{U}_{n, n}$ satisfies Theorem \ref{main}.

\vspace{.3cm}
\noindent {\bf The organization of the paper:}
In Section 2 we construct the quivers $\gn, \gnn$ and the $q$-graded DG algebras $\rn, \rnn$ and $\drnn$.
In Section 3 we define the multiplication on $K_0(\rn)$ and show that it is associative.
In Section 4 we give a categorification of the multiplication: $\mathcal{M}_n: DGP(\drnn) \ra DGP(\rn)$ through the $q$-graded DG $(\rn,\rnn)$-bimodules $T_n$.
In Section 5 we construct the subcategory $\mathcal{U}_n$.
It categorifies the integral version $\mf{U}_n$ of $\uq$ and the restriction of $\mathcal{M}_n|_{H^0}$ categorifies the algebra structure on $\mf{U}_n$.

\vspace{.3cm}
\noindent {\bf Acknowledgements:}
I am very grateful to Ko Honda for many ideas and suggestions and introducing me to the contact category.
I would like to thank Aaron Lauda for teaching me a great deal about the representation theory and categorification.
Thanks also to Yang Huang for useful discussions.

\section{The $q$-graded DG algebras $\rn$ and $\drnn$}
\subsection{The $q$-graded DG algebra $\rn$}
\subsubsection{The quiver $\gn$}
In this section we construct a family of quivers $\gn=(V(\gn), A(\gn))$ for $n>0$, where $V(\gn)$ and $A(\gn)$ are vertex and arrow sets of $\gn$.

\begin{defn} [Quiver $\gn=(V(\gn), A(\gn))$]
\n \be
\item Let $V(\gn)$ be the set of decreasing sequences of integers bounded by $n$ and $0$, i.e.,

$V(\gn) = \{[\emptyset]\} \sqcup \{\mf{x} = [x_0,\cdots,x_i] ~ |~ n\geq x_0 > \cdots > x_i \geq 0, x_k \in \Z ~\mbox{for}~ 0 \leq k \leq i\}.$

\item Let $l:V(\gn)\rightarrow \Z$ be a {\em length function}, given by $l(\mf{x})=i+1$ for $\mf{x}= [x_0,x_1,\cdots,x_i]$ and $l([\es])=0$.

\item Let $A(\gn)$ be the subset of $V(\gn) \times V(\gn)$ ,
where $(\mf{x}, \mf{y}) \in A(\gn)$, $\mf{x}= [x_0,x_1,\cdots,x_i]$, $\mf{y}= [y_0,y_1,\cdots,y_j]$ if $j = i + 2$ and as sets $\{y_0,y_1,\cdots,y_j\}=\{x_0,x_1,\cdots,x_i\} \sqcup \{s+1, s\}$ for some $s$, i.e., $\mf{y}$ is obtained from $\mf{x}$ by adding a pair of adjacent integers $\{s+1, s\}$.
We write an arrow $(\mf{x} \stackrel{s}\to \mf{y})$ if $(\mf{x}, \mf{y}) \in A(\gn)$.
\ee
\end{defn}

\begin{rmk}
The vertices of $\gn$ are in one-to-one correspondence with subsets of $\{0,1,\cdots,n\}$.
We write $s \in \mf{x} = [x_0,x_1,\cdots,x_i]$ if $s=x_k$ for some $k$ and $s \notin \mf{x}$ otherwise.
\end{rmk}

In order to decompose $\gn$ into its connected components, we define the following:

\begin{defn} [Euler grading on $V(\gn)$]
The {\em Euler grading} $\op{e}: V(\gn) \ra \Z$ is defined as

$\op{e}(\mf{x})=\sum \limits_{k=0}^{i} (-1)^{x_k}$ for $\mf{x} = [x_0,x_1,\cdots,x_i]$ and
$\op{e}([\es])=0$.
\end{defn}

It is easy to see that $\mf{x}$ and $\mf{y}$ are in the same connected component of $\gn$ if and only if they have the same Euler grading: $\op{e}(\mf{x})= \op{e}(\mf{y})$.
Therefore, $\gn = \sqcup_e \Gamma_{n,e}$, where $\Gamma_{n,e}$ is the connected component with Euler grading $e$.

\begin{rmk}
The Euler grading $\op{e}$ actually comes from the Euler number of a dividing set.
Recall a dividing set divides the surface into positive and negative regions.
Then the {\em Euler number} is the Euler characteristic of the positive region minus the Euler characteristic of the negative region.
\end{rmk}

\begin{example} [Quiver $\Gamma_2$]
The quiver $\Gamma_2$ has four components $\Gamma_{2,e}$ for $e=-1, 0 ,1, 2$, where $\Gamma_{2,0}$ and $\Gamma_{2, 1}$ are dual to each other. See Figure \ref{quiver}.
\end{example}
\begin{figure}[h]
\begin{overpic}
[scale=0.3]{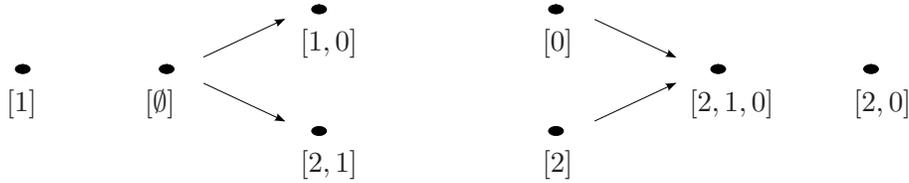}
\put(-1,3){$[1]$}
\put(97,3){$[2,0]$}
\put(15,3){$[\es]$}
\put(78,3){$[2,1,0]$}
\put(33,10){$[1,0]$}
\put(61,10){$[0]$}
\put(33,-4){$[2,1]$}
\put(61,-4){$[2]$}
\end{overpic}
\caption{The quiver $\Gamma_2$.}
\label{quiver}
\end{figure}

\subsubsection{The $q$-graded algebra $\rn$}
We define the $q$-graded algebra $\rn$ which is closely related to the path algebra $\F \gn$ of the quiver $\gn$. $\F$ is fixed as the ground field throughout the paper.

\begin{defn} [$q$-graded algebra $\rn$] \label{def rn}
$\rn$ is an associative graded $\F$-algebra with generators $r(\mf{x} \stackrel{s}\to \mf{y})$ for each arrow $(\mf{x} \xrightarrow{s} \mf{y})$ in $\gn$, idempotents $e(\mf{x})$ for each vertex $\mf{x}$ in $\gn$ and relations:
\begin{gather*}
e(\mf{x}) \cdot e(\mf{y})=\delta_{\mf{x},\mf{y}}\cdot e(\mf{x}) ~\mbox{for all}~ \mf{x}, \mf{y};\\
e(\mf{x}) \cdot r(\mf{x} \stackrel{s}\to \mf{y})= r(\mf{x} \stackrel{s} \to \mf{y}) \cdot e(\mf{y})= r(\mf{x} \stackrel{s} \to \mf{y}) ~\mbox{for all}~ r(\mf{x} \stackrel{s}\to \mf{y});\\
r(\mf{x} \stackrel{s} \to \mf{y}) \cdot r(\mf{y} \stackrel{t} \to \mf{w})=r(\mf{x} \stackrel{t} \to \mf{z}) \cdot r(\mf{z} \stackrel{s} \to \mf{w})  ~\mbox{for all}~ r(\mf{x} \stackrel{s} \to \mf{y}), r(\mf{y} \stackrel{t} \to \mf{w}), r(\mf{x} \stackrel{t} \to \mf{z}), r(\mf{z} \stackrel{s} \to \mf{w}).
\end{gather*}
The unit of $\rn$ is $\mathbf{1}_{\rn}=\sum \limits_{\mf{x} \in V(\gn)} e(\mf{x})$.
The $q$-grading on $\rn$ is given on generators as follows: $\op{deg}_{\rn}(e(\mf{x}))=0, \op{deg}_{\rn}(r(\mf{x} \stackrel{s} \to \mf{y}))=n-1-2s$.
\end{defn}

\begin{rmk}
The last relation in the definition above means that the two paths from $\mf{x}$ to $\mf{w}$ in $\gn$: $\mf{x} \stackrel{s} \to \mf{y} \stackrel{t} \to \mf{w}$ and $\mf{x} \stackrel{t} \to \mf{z} \stackrel{s} \to \mf{w}$ commute if $\mf{w}$ is obtained from $\mf{x}$ by adding two disjoint pairs of adjacent integers $\{s+1, s\}$ and $\{t+1, t\}$.
\end{rmk}

We refer to the book \cite{ASS} for an introduction to the representation theory of quivers.
It is easy to see $\rn$ is a finite dimensional algebra since the quiver $\gn$ has no oriented cycles.
Since $\{e(\mf{x})~|~\mf{x}\in \gn \}$ is a complete set of primitive orthogonal idempotents in $\rn$, $\{P(\mf{x})=\rn e(\mf{x}) ~|~ \mf{x} \in \gn \}$ forms a complete set of non-isomorphic indecomposable projective graded left $\rn$-modules, up to grading shifts.
Let $A\{m\}$ denote $A$ with its grading shifted by $m$, i.e., $A\{m\}=\{a \in A~|~ \op{deg}_{A\{m\}}(a)=\op{deg}_{A}(a)-m\}$.
Then any projective graded left $\rn$-module $A$ is a direct sum of indecomposables $P(\mf{x})\{m\}$.

Consider $\mf{K}^{b}(\rn)$, the homotopy category of bounded cochain complexes of finitely generated projective graded modules over $\rn$ with grading-preserving differentials.
The chain maps are also grading-preserving.
For any cochain complex $M=\{\cdots \to M^s \to M^{s+1} \to \cdots \} \in \mf{K}^{b}(\rn)$, let $M[p]$ be $M$ with the homological grading shifted by $p$, i.e., $M[p]^s=M^{s+p}$.
By a standard result in homological algebra, $\mf{K}^{b}(\rn)$ is a triangulated category.

Let $K_0(\rn)$ be the Grothendieck group of $\mf{K} ^{b}(\rn)$.
It is a $\Z[q^{\pm1}]$-module with generators $[P]$ over all finitely generated projective graded $\rn$-modules and relations $[P\{1\}]=q[P]$, $[P[1]]=-[P]$ and $[P_2]=[P_1]+[P_3]$ for each short exact sequence $0 \to P_1 \to P_2 \to P_3 \to 0$.
It is easy to see $K_0(\rn)$ is a free $\Z[q^{\pm1}]$-module over the basis $\{[P(\mf{x})]~|~ \mf{x} \in V(\gn)\}$.
If $\Z[q^{\pm1}]\langle S \rangle$ denotes the free $\Z[q^{\pm1}]$-module generated by the set $S$, then
$$K_0(\rn) = \Z[q^{\pm1}]\langle V(\gn)\rangle.$$

\subsubsection{The DG category $DGP(\rn)$}
We refer to \cite[Section 10]{BL} for an introduction to DG algebras and DG modules.
We consider DG algebras with an additional $q$-grading as follows.
\begin{defn}
A {\em $q$-graded DG algebra $(A, d)$} is a doubly graded associative $\F$-algebra $A=\bigoplus\limits_{i,j} A^{i,j}$ with a unit $\mf{1}_A \in A^{0,0}$, where $i$ is the cohomological grading and $j$ is the $q$-grading.
The differential $d$ is an additive endomorphism of degree $(1,0)$ such that for $a, b \in A$:
\begin{gather*}
d^2=0,  ~ d(\mf{1}_A)=0, \\
d(a\cdot b)=d(a)\cdot b + a \cdot d(b).
\end{gather*}
\end{defn}

\begin{defn}
A {\em left $q$-graded DG module $(M, d_M)$} over a $q$-graded DG algebra $(A, d)$ is a doubly graded unitary left $A$-module $M=\bigoplus\limits_{i,j} M^{i,j}$, where $i$ is the cohomological grading and $j$ is the $q$-grading. The differential
$d_M$ is an additive endomorphism of degree $(1,0)$ such that for $a \in A, m \in M$
\begin{align*}
d_M^2=0, \hspace{.5cm} d_M(a\cdot m)=d(a)\cdot m + a \cdot d_M(m).
\end{align*}
\end{defn}

We view $\rn$ as a $q$-graded DG algebra $(\rn, d=0)$ which has trivial differential and is concentrated in cohomological grading $0$.
We are interested in the differential graded category $DG(\rn)$ of $q$-graded DG $\rn$-modules.
We refer to \cite{Ke} for an introduction to differential graded categories.
\begin{defn} [DG category $DG(A)$ for a $q$-graded DG algebra $A$]
\n \be
\item The objects of $DG(A)$ are left $q$-graded DG $A$-modules.
\item The space of morphisms $(\op{Hom}_{DG(A)}(M, N), d)=(\bigoplus \limits_{i}\op{Hom}_{DG(A)}^{i}(M, N), \bigoplus \limits_{i}d^i)$ is a cochain complex, where $\op{Hom}_{DG(A)}^{i}(M, N)$ is the set of left $A$-module maps of degree $(i, 0)$ and $$d^i: \op{Hom}_{DG(A)}^{i}(M, N) \ra \op{Hom}_{DG(A)}^{i+1}(M, N)$$ is given as $d(f)=d \circ f + f \circ d$.
\item A morphism $f$ is {\em closed} if $d(f)=0$; $f$ is {\em exact} if $f=d(g)$ for some morphism $g$.
    Let $Z^{i}(\op{Hom}_{DG(A)}(M, N))$ and $B^{i}(\op{Hom}_{DG(A)}(M, N))$ denote the subset of $\op{Hom}_{DG(A)}^i(M, N)$ consisting of closed morphisms and exact morphisms respectively.
\ee
\end{defn}

\rmk For $(\rn, d=0)$, a $q$-graded DG $\rn$-module is a cochain complex of $q$-graded $\rn$-modules. A closed morphism of degree $(0, 0)$ is a chain map.

\begin{defn}
The {\em 0-th homology category} $H^0(DG(A))$ of the DG category $DG(A)$ has the same objects as $DG(A)$ and its morphisms are given by $$\op{Hom}_{H^0(DG(A))}(M, N)=Z^{0}(\op{Hom}_{DG(A)}(M, N))/B^{0}(\op{Hom}_{DG(A)}(M, N)).$$
\end{defn}

\begin{rmk} The 0-th homology category $H^0(DG(A))$ is isomorphic to the homotopy category of $q$-graded DG $A$-modules.
\end{rmk}

\begin{defn}
A DG $A$-module $P$ is called {\em projective} if the complex $\op{Hom}_{DG(A)}(P, M)$ has zero cohomology when the cohomology $H(M)$ of $(M, d_M) \in Ob(DG(A))$ is zero.
\end{defn}

\begin{rmk} The DG $\rn$-module $P(\mf{x})=\rn e(\mf{x})$ is projective since it is a direct summand of $\rn$ which is projective \cite[Remark 10.12.2.3]{BL}.
\end{rmk}

There are two automorphisms of $DG(A)$: $[1]$ and $\{1\}$ with respect to the  cohomological grading and the $q$-grading.
There is another operation, called the {\em mapping cone}, which constructs a new object $C(f)$ from $f \in Z^{0}(\op{Hom}_{DG(A)}(M, N))$
\begin{defn} [Two shift functors and $C(f)$]
\n \be
\item The shift functor $[1]: DG(A) \ra DG(A)$ is an automorphism of $DG(A)$ such that,
    $$(M[1])^{i,j}=M^{i+1,j}, ~~~d_{M[1]}=d_M.$$
\item The shift functor $\{1\}: DG(A) \ra DG(A)$ is an automorphism of $DG(A)$ such that,
    $$(M\{1\})^{i,j}=M^{i,j+1}, ~~~d_{M\{1\}}=d_M.$$
\item For $f \in Z^{0}(\op{Hom}_{DG(A)}(M, N))$, $M, N \in Ob(DG(A))$, define the mapping cone $C(f)=N\oplus M[1]$ with the differential $d_{C(f)}=(d_N + f, d_M)$.
\ee
\end{defn}

\begin{defn}
Let $DGP(\rn)$ be the smallest full subcategory of $DG(\rn)$ which contains the projective DG $\rn$-modules $\{P(\mf{x}) ~|~ \mf{x} \in V(\gn)\}$ and is closed under the two shift functors $[1], \{1\}$ and taking the mapping cones.
\end{defn}

In other words, the objects of $DGP(\rn)$ are finitely iterated cones of closed morphisms between the projective modules $\{P(\mf{x}) ~|~ \mf{x} \in V(\gn)\}$ up to grading shifts.
Since $\{P(\mf{x}) ~|~ \mf{x} \in \gn \}$ form a complete set of non-isomorphic indecomposable projective $\rn$-modules up to grading shifts, the $0$-th homology category $H^0(DGP(\rn))$ is equivalent to $\mf{K}^{b}(\rn)$ as triangulated categories.
Hence their Grothendieck groups are isomorphic:
$$K_0(H^0(DGP(\rn))) \cong K_0(\mf{K}^{b}(\rn)) = K_0(\rn).$$

\subsection{The $q$-graded DG algebra $\drnn$}

\subsubsection{The $q$-graded algebra $\rnn$}
To categorify $\op{m}_n: K_0(\rn) \otimes_{\Z[q^{\pm1}]} K_0(\rn) \rightarrow K_0(\rn)$, we look for a category whose Grothendieck group is $K_0(\rn) \otimes_{\Z[q^{\pm1}]} K_0(\rn)$.

\begin{defn} [$q$-graded algebra $\rnn$]
As an $\F$-algebra, $\rn \otimes_{\F} \rn$ is the tensor product of two $\rn$'s over $\F$ with unit
$$\mathbf{1}_{\rn \otimes_{\F} \rn}=\sum \limits_{\mf{x},\mf{y} \in V(\gn)} e(\mf{x}) \otimes e(\mf{y}).$$
The $q$-grading on generators is given as $\op{deg}_{\rn \otimes_{\F} \rn}(a \otimes b) = \op{deg}_{\rn}(a) + \op{deg}_{\rn}(b)$.
\end{defn}

For simplicity, we omit ground rings or fields in various tensor products.
For instance, we write $\rn \otimes \rn$ for $\rn \otimes_{\F} \rn$ and $K_0(\rn) \otimes K_0(\rn)$ for $K_0(\rn) \otimes_{\Z[q^{\pm1}]} K_0(\rn)$.

Since $\{e(\mf{x}) \otimes e(\mf{y}) ~|~ \mf{x},\mf{y} \in V(\gn)\}$ is a complete set of primitive orthogonal idempotents of $\rnn$, the modules $P'(\mf{x,y})=(\rnn) (e(\mf{x}) \otimes e(\mf{y}))=\rn e(\mf{x}) \otimes \rn e(\mf{y})$ form a complete set of non-isomorphic indecomposable projective graded left $\rnn$-modules, up to grading shifts.
Consider $\mf{K}^{b}(\rnn)$, the homotopy category of bounded cochain complexes of finitely generated projective graded modules over $\rnn$ with grading-preserving differentials and chain maps.
Let $K_0(\rnn)$ be the Grothendieck group of $\mf{K}^{b}(\rnn)$.
It is easy to see that $K_0(\rnn)$ is a free $\Z [q^{\pm1}]$-module over $\{[P'(\mf{x,y})]~|~ \mf{x}, \mf{y} \in V(\gn)\}$, i.e.,
$$K_0(\rnn) = \Z [q^{\pm1}]\lan V(\gn) \times V(\gn)\ran = \Z [q^{\pm1}]\lan V(\gn)\ran \otimes \Z [q^{\pm1}]\lan V(\gn)\ran.$$
Hence we have
$K_0(\rnn)=K_0(\rn) \otimes K_0(\rn).$

\subsubsection{The $q$-graded DG-algebra $\drnn$}
We construct a family of quivers $\gnn$ viewed as a variant of the ``product" of two $\gn$'s.
\begin{defn} [Quiver $\gnn=(V(\gnn), A(\gnn))$]
\n \be
\item $V(\gnn) = V(\gn) \times V(\gn)$.

\item Let $A(\gnn)$ be the subset of
$$V(\gnn) \times V(\gnn)=(V(\gn) \times V(\gn)) \times (V(\gn) \times V(\gn)),$$
where $(\mf{x}, \mf{y}, \mf{x'}, \mf{y'}) \in A(\gnn)$ if one of following holds:
\be
\item $(\mf{x}, \mf{x'}) \in A(\gn)$ and $\mf{y}=\mf{y'}$;
\item $(\mf{y}, \mf{y'}) \in A(\gn)$ and $\mf{x}=\mf{x'}$;
\item $(\mf{x}, \mf{x'}), (\mf{y}, \mf{y'}) \in A(\gn)$ and there exist some $s \in \{0,1...,n-1\}$ such that the corresponding arrows are $(\mf{x} \stackrel{s}\to \mf{x'})$ and $(\mf{y} \stackrel{s+1}\to \mf{y'})$
\ee
We denote the arrows for (a), (b) and (c) by $(\mf{x},\mf{y} \xrightarrow{s,\es} \mf{x'},\mf{y})$,
$(\mf{x},\mf{y} \xra{\es, s} \mf{x},\mf{y'})$ and
$(\mf{x},\mf{y} \xra{s, s+1} \mf{x'},\mf{y'})$ respectively.
\ee
\end{defn}

\begin{example} [Quiver $\Gamma_2 \boxtimes \Gamma_2$]
One component of $\Gamma_2 \boxtimes \Gamma_2$ with the arrow $([\es],[\es] \xra{0, 1} [1,0],[2,1])$ is shown in Fig \ref{quiver'}.

\begin{figure}[h]
\begin{overpic}
[scale=0.3]{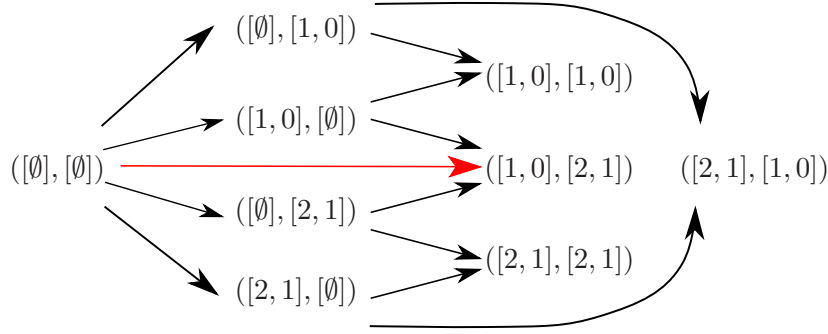}
\put(-15,25){$\tiny{([\es], [\es])}$}
\put(22,48){$\tiny{([\es], [1,0])}$}
\put(22,33){$\tiny{([1,0], [\es])}$}
\put(22,18){$\tiny{([\es], [2,1])}$}
\put(22,5){$\tiny{([2,1], [\es])}$}
\put(63,40){$\tiny{([1,0], [1,0])}$}
\put(63,25){$\tiny{([1,0], [2,1])}$}
\put(63,10){$\tiny{([2,1], [2,1])}$}
\put(95,25){$\tiny{([2,1], [1,0])}$}
\end{overpic}
\caption{One component of the quiver $\Gamma_2 \boxtimes \Gamma_2$.}
\label{quiver'}
\end{figure}
\end{example}

The algebra $\drnn$ is obtained by adding a differential to some quotient of the path algebra $\F(\gnn)$ of the quiver $\gnn$.

\begin{defn} [$q$-graded DG algebra $\drnn$] \label{DG algebra}
$(\drnn, d)$ is an associative $q$-graded DG algebra with a differential $d$, a cohomological grading and a $q$-grading.

\noindent (A) $\drnn$ has generators $r(\mf{x},\mf{y} \xra{s,t} \mf{x'},\mf{y'})$ for each arrow $(\mf{x},\mf{y} \xra{s,t} \mf{x'},\mf{y'})$ in $\gnn$, here $s$ or $t$ maybe $\es$, idempotents $e(\mf{x},\mf{y})$ for each vertex $(\mf{x},\mf{y})$ in $\gnn$ and relations:
\begin{gather}
e(\mf{x},\mf{y}) \cdot e(\mf{x'},\mf{y'})=\delta_{\mf{x},\mf{x'}}\cdot \delta_{\mf{y},\mf{y'}}\cdot e(\mf{x},\mf{y});\\
e(\mf{x},\mf{y}) \cdot r(\mf{x},\mf{y} \xra{s,t} \mf{x'},\mf{y'})= r(\mf{x},\mf{y} \xra{s,t} \mf{x'},\mf{y'}) \cdot e(\mf{x'},\mf{y'})= r(\mf{x},\mf{y} \xra{s,t} \mf{x},\mf{y'});\\
r(\mf{x},\mf{y} \xra{\es,s} \mf{x},\mf{y'})\cdot r(\mf{x},\mf{y'} \xra{\es,t} \mf{x},\mf{y'''})=
r(\mf{x},\mf{y} \xra{\es,t} \mf{x},\mf{y''})\cdot r(\mf{x},\mf{y''} \xra{\es,s} \mf{x},\mf{y'''});\\
r(\mf{x},\mf{y} \xra{s,\es} \mf{x'},\mf{y})\cdot r(\mf{x'},\mf{y} \xra{t,\es} \mf{x'''},\mf{y})=
r(\mf{x},\mf{y} \xra{t,\es} \mf{x''},\mf{y})\cdot r(\mf{x''},\mf{y} \xra{s,\es} \mf{x'''},\mf{y}); \\
r(\mf{x},\mf{y} \xra{s,\es} \mf{x'},\mf{y})\cdot r(\mf{x'},\mf{y} \xra{\es,t} \mf{x'},\mf{y'})=
r(\mf{x},\mf{y} \xra{\es,t} \mf{x},\mf{y'})\cdot r(\mf{x},\mf{y'} \xra{s,\es} \mf{x'},\mf{y'}) ~~\mbox{if}~~t\neq s+1; \\
r(\mf{x},\mf{y} \xra{s,s+1} \mf{x'},\mf{y'})\cdot r(\mf{x'},\mf{y'} \xra{\es,t} \mf{x'},\mf{y'''})=
r(\mf{x},\mf{y} \xra{\es,t} \mf{x},\mf{y''})\cdot r(\mf{x},\mf{y''} \xra{s,s+1} \mf{x'},\mf{y'''}); \\
r(\mf{x},\mf{y} \xra{s,s+1} \mf{x'},\mf{y'})\cdot r(\mf{x'},\mf{y'} \xra{t, \es} \mf{x'''},\mf{y'})=
r(\mf{x},\mf{y} \xra{t, \es} \mf{x''},\mf{y})\cdot r(\mf{x''},\mf{y} \xra{s,s+1} \mf{x'''},\mf{y'}); \\
r(\mf{x},\mf{y} \xra{s,s+1} \mf{x'},\mf{y'})\cdot r(\mf{x'},\mf{y'} \xra{t,t+1} \mf{x'''},\mf{y'''})=
r(\mf{x},\mf{y} \xra{t,t+1} \mf{x''},\mf{y''})\cdot r(\mf{x''},\mf{y''} \xra{s,s+1} \mf{x'''},\mf{y'''}).
\end{gather}
The unit of $\drnn$ is $\mathbf{1}_{\drnn}=\sum \limits_{(\mf{x},\mf{y}) \in V(\gnn)} e(\mf{x}, \mf{y})$.

\vspace{.1cm}
\noindent (B) The differential $d$ is given on the generators as:
$$
d(r(\mf{x},\mf{y} \xra{s,t} \mf{x'},\mf{y'}))=\left\{
\begin{array}{cl}
r(\mf{x},\mf{y} \xra{s,\es} \mf{x'},\mf{y})\cdot r(\mf{x'},\mf{y} \xra{\es,s+1} \mf{x'},\mf{y'})& \\
+
r(\mf{x},\mf{y} \xra{\es,s+1} \mf{x},\mf{y'})\cdot r(\mf{x},\mf{y'} \xra{s,\es} \mf{x'},\mf{y'}) & ~~\mbox{if}~~ t=s+1,\\
0 & ~~\mbox{otherwise};
\end{array}\right.
$$
and extended by $d(r_1\cdot r_2)=d(r_1)\cdot r_2+r_1\cdot d(r_2)$ for generators $r_1, r_2$.

\vspace{.1cm}
\noindent (C) The cohomological grading $\op{gr}$ is given on generators as: $\op{gr}(r(\mf{x},\mf{y} \xra{s,s+1} \mf{x'},\mf{y'}))=-1$ and $\op{gr}(r(\mf{x},\mf{y} \xra{s,t} \mf{x'},\mf{y'}))=0$ otherwise.

\vspace{.1cm}
\noindent (D) The $q$-grading $\op{deg}$ is given on generators as: $\op{deg}(r(\mf{x},\mf{y} \xra{s,t} \mf{x'},\mf{y'}))=(n-1-2s)+(n-1-2t)$ and $\op{deg}(e(\mf{x}, \mf{y}))=0$.
\end{defn}

\begin{rmk}
Relations (1)--(5) of $\drnn$ are analogous to those of $\rnn$ if we identify $e(\mf{x},\mf{y}), r(\mf{x},\mf{y} \xra{\es,s} \mf{x},\mf{y'}), r(\mf{x},\mf{y} \xra{t,\es} \mf{x'},\mf{y})$ with $e(\mf{x})\ot e(\mf{y}), e(\mf{x})\ot r(\mf{y} \xra{s} \mf{y'}), r(\mf{x} \xra{t} \mf{x'}) \ot e(\mf{y})$.
\end{rmk}

\begin{lemma} $d$ is a differential on $\drnn$.
\end{lemma}
\begin{proof}
It suffices to prove that $d$ preserves Relations (6), (7) and (8) in Definition \ref{DG algebra} (A).
We prove (6), leaving (7) and (8) to the reader.

Applying $d$ to the left-hand side of (6), we have:
\begin{align*}
& d(r(\mf{x},\mf{y} \xra{s,s+1} \mf{x'},\mf{y'})\cdot r(\mf{x'},\mf{y'} \xra{\es,t} \mf{x'},\mf{y'''})) \\
= & d(r(\mf{x},\mf{y} \xra{s,s+1} \mf{x'},\mf{y'}))\cdot r(\mf{x'},\mf{y'} \xra{\es,t} \mf{x'},\mf{y'''}) \\
= & r(\mf{x},\mf{y} \xra{s,\es} \mf{x'},\mf{y})\cdot r(\mf{x'},\mf{y} \xra{\es,s+1} \mf{x'},\mf{y'})\cdot r(\mf{x'},\mf{y'} \xra{\es,t} \mf{x'},\mf{y'''}) \\
& + r(\mf{x},\mf{y} \xra{\es,s+1} \mf{x},\mf{y'})\cdot r(\mf{x},\mf{y'} \xra{s,\es} \mf{x'},\mf{y'})\cdot r(\mf{x'},\mf{y'} \xra{\es,t} \mf{x'},\mf{y'''})\\
= & r(\mf{x},\mf{y} \xra{s,\es} \mf{x'},\mf{y})\cdot r(\mf{x'},\mf{y} \xra{\es,t} \mf{x'},\mf{y''})\cdot r(\mf{x'},\mf{y''} \xra{\es,s+1} \mf{x'},\mf{y'''}) \\
& + r(\mf{x},\mf{y} \xra{\es,s+1} \mf{x},\mf{y'})\cdot r(\mf{x},\mf{y'} \xra{s,\es} \mf{x'},\mf{y'})\cdot r(\mf{x'},\mf{y'} \xra{\es,t} \mf{x'},\mf{y'''})
\end{align*}
We use Relation (5) of commutativity:
\begin{gather*}
r(\mf{x},\mf{y} \xra{s,\es} \mf{x'},\mf{y})\cdot r(\mf{x'},\mf{y} \xra{\es,t} \mf{x'},\mf{y''})=
r(\mf{x},\mf{y} \xra{\es,t} \mf{x},\mf{y''})\cdot r(\mf{x},\mf{y''} \xra{s,\es} \mf{x'},\mf{y''}) \\
r(\mf{x},\mf{y'} \xra{s,\es} \mf{x'},\mf{y'})\cdot r(\mf{x'},\mf{y'} \xra{\es,t} \mf{x'},\mf{y'''})=
r(\mf{x},\mf{y'} \xra{\es,t} \mf{x},\mf{y'''})\cdot r(\mf{x},\mf{y'''} \xra{s,\es} \mf{x'},\mf{y'''})
\end{gather*}
since $t \neq s+1$. Then we have
\begin{align*}
& d(r(\mf{x},\mf{y} \xra{s,s+1} \mf{x'},\mf{y'})\cdot r(\mf{x'},\mf{y'} \xra{\es,t} \mf{x'},\mf{y'''})) \\
= & r(\mf{x},\mf{y} \xra{\es,t} \mf{x},\mf{y''})\cdot r(\mf{x},\mf{y''} \xra{s,\es} \mf{x'},\mf{y''})\cdot r(\mf{x'},\mf{y''} \xra{\es,s+1} \mf{x'},\mf{y'''}) \\
& + r(\mf{x},\mf{y} \xra{\es,s+1} \mf{x},\mf{y'})\cdot r(\mf{x},\mf{y'} \xra{\es,t} \mf{x},\mf{y'''})\cdot r(\mf{x},\mf{y'''} \xra{s,\es} \mf{x'},\mf{y'''})
\end{align*}
\begin{align*}
= & r(\mf{x},\mf{y} \xra{\es,t} \mf{x},\mf{y''})\cdot r(\mf{x},\mf{y''} \xra{s,\es} \mf{x'},\mf{y''})\cdot r(\mf{x'},\mf{y''} \xra{\es,s+1} \mf{x'},\mf{y'''}) \\
& + r(\mf{x},\mf{y} \xra{\es,t} \mf{x},\mf{y''})\cdot r(\mf{x},\mf{y''} \xra{\es,s+1} \mf{x},\mf{y'''})\cdot r(\mf{x},\mf{y'''} \xra{s,\es} \mf{x'},\mf{y'''})\\
= & r(\mf{x},\mf{y} \xra{\es,t} \mf{x},\mf{y''})\cdot d(r(\mf{x},\mf{y''} \xra{s,s+1} \mf{x'},\mf{y'''})) \\
= & d(r(\mf{x},\mf{y} \xra{\es,t} \mf{x},\mf{y''})\cdot r(\mf{x},\mf{y''} \xra{s,s+1} \mf{x'},\mf{y'''})),
\end{align*}
which is the differential of the right-hand side of (6).
\end{proof}

\begin{defn}
Let $DGP(\drnn)$ be the smallest full subcategory of $DG(\drnn)$ which contains the projective DG $\drnn$-modules $\{P(\mf{x}, \mf{y})=(\drnn)e(\mf{x}, \mf{y}) ~|~ \mf{x}, \mf{y} \in V(\gn)\}$ and is closed under the two shift functors $[1], \{1\}$ and taking the mapping cones.
\end{defn}

In order to understand the $0$-th homology category $H^0(DGP(\drnn))$, we look at the relation between $\drnn$ and its cohomology algebra $H(\drnn)$.
\begin{defn}
A $q$-graded DG algebra $A$ is {\em formal} if it is $q$-graded quasi-isomorphic to its cohomology $H(A)$.
\end{defn}

\begin{lemma} The $q$-graded DG algebra $\drnn$ is formal and its cohomology $H(\drnn)$ is isomorphic to $\rnn$.
\end{lemma}
\proof
It is easy to see that the cohomology $H(\drnn)$ is isomorphic to $\rnn$.
We define a quasi-isomorphism $H:  \drnn  \ra  \rnn$ as follows:
$$
\begin{array}{cccr}
 e(\mf{x},\mf{y}) & \mapsto & e(\mf{x}) \otimes e(\mf{y})& \\
   r(\mf{x},\mf{y} \xra{\es,s} \mf{x},\mf{y'}) & \mapsto & e(\mf{x}) \otimes r(\mf{y}\xra{s} \mf{y'})& \\
   r(\mf{x},\mf{y} \xra{s,\es} \mf{x'},\mf{y}) & \mapsto & r(\mf{x}\xra{s} \mf{x'}) \otimes e(\mf{y})& \\
   r(\mf{x},\mf{y} \xra{s,s+1} \mf{x'},\mf{y'}) & \mapsto & 0 &\qed
\end{array} $$

We also view $\rnn$ as a $q$-graded DG algebra $(\rnn, d=0)$ which has trivial differential and is concentrated in cohomological grading $0$.
\begin{defn}
Let $DGP(\rnn)$ be the smallest full subcategory of $DG(\rnn)$ which contains the projective DG $\rnn$-modules $\{P'(\mf{x}, \mf{y}) ~|~ \mf{x}, \mf{y} \in V(\gn)\}$ and is closed under the two shift functors $[1], \{1\}$ and taking the mapping cones.
\end{defn}

Let $Z^0(DG(A))$ be an abelian category with the same objects as $DG(A)$, whose morphisms are given by
$$\op{Hom}_{Z^0(DG(A))}(M, N)=Z^{0}(\op{Hom}_{DG(A)}(M, N)).$$
Consider the homotopy category $KDG(A)$ and derived category $DDG(A)$ of $Z^0(DG(A))$.
They are triangulated categories.
Let $KPDG(A)$ be the full subcategory of $KDG(A)$ consisting of projective $q$-graded DG $A$-modules.
The localization functor induces an equivalence: $KPDG(A) \ra DDG(A)$ \cite[Corollary 10.12.2.9]{BL}.

For any quasi-isomorphism $F: A \ra B$ of DG algebras, the derived induction functor $ind=B\bigotimes\limits^{L}$$_A-: DDG(A) \ra DDG(B)$ is an equivalence of categories \cite[Theorem 10.12.5.1]{BL}.
The induced functor $ind$ on $KPDG(A)$ is the induction functor $B\otimes_A-$.
It maps any projective DG $A$-module $P$ to a projective DG $B$-module $B \otimes_A P$ since $$\op{Hom}_{DG(B)}(B\otimes_A P, N)\cong \op{Hom}_{DG(A)}(P, Res(N))$$
has zero cohomology for any DG $B$-module $N$ with $H(N)=0$, where $Res(N)$ is the restriction of $N$ as a $A$-module.
Hence we have an equivalence $B\otimes_A-: KPDG(A) \ra KPDG(B)$.
In particular, $ind_n: KPDG(\drnn) \ra KPDG(\rnn)$ is an equivalence since $\drnn$ is quasi-isomorphic to $\rnn$.
We have the following equivalence for their subcategories:
\begin{lemma} \label{equivalence}
The $0$-th homology category $H^0(DGP(\drnn))$ is equivalent to the $0$-th homology category $H^0(DGP(\rnn))$.
\end{lemma}
\begin{proof}
Notice that $ind_n(P(\mf{x}, \mf{y}))=P'(\mf{x}, \mf{y})$ and $H^0(DGP(\drnn))$ is a full subcategory of $KPDG(\drnn)$.
We have a restriction of $ind_n$: $H^0(DGP(\drnn)) \ra H^0(DGP(\rnn))$.
It is fully faithful since $ind_n: KPDG(\drnn) \ra KPDG(\rnn)$ is an equivalence.

Any object $N$ in $H^0(DGP(\rnn))$ is a finitely iterated cone of maps between $P'(\mf{x}, \mf{y})$'s, hence it is isomorphic to an object $ind_n(M)$ for some $M$ in $H^0(DGP(\drnn))$.
Therefore, the restriction of $ind_n: H^0(DGP(\drnn)) \ra H^0(DGP(\rnn))$ induces an equivalence of categories.
\end{proof}

Since $\{P'(\mf{x}, \mf{y}) ~|~ \mf{x}, \mf{y} \in \gn \}$ form a complete set of non-isomorphic indecomposable projective $\rnn$-modules up to grading shifts, the $0$-th homology category $H^0(DGP(\rnn))$ is equivalent to $\mf{K}^{b}(\rnn)$.
\begin{cor}
There are isomorphisms of the Grothendieck groups:
$$K_0(H^0(DGP(\drnn))) \cong K_0(H^0(DGP(\rnn))) \cong K_0(\mf{K}^{b}(\rnn)) = K_0(\rnn).$$
\end{cor}

\section{The multiplication on $K_0(\rn)$}
The goal of this section is to define the multiplication $\op{m}_n: K_0(\rnn) \rightarrow K_0(\rn)$ and show that it is associative.
We fix some $n>0$ throughout this and next sections and omit the subscript $n$.

We first construct a $\Z [q^{\pm1}, h^{\pm1}]$-bilinear map
$$\op{M}: \Z [q^{\pm1}, h^{\pm1}]\lan V(\g) \ran \times \Z [q^{\pm1}, h^{\pm1}]\lan V(\g) \ran \ra \Z [q^{\pm1}, h^{\pm1}]\lan V(\g) \ran,$$
where $h$ corresponds to the cohomological grading in $DGP(R)$ or $\mf{K}^{b}(R)$.
We call $\op{M}$ the {\em higher multiplication}.

\subsection{The higher multiplication $\op{M}$}

Given any pair of decreasing sequences $\mf{x}= [x_0,x_1,\cdots,x_i]$ and $\mf{y}= [y_0,y_1,\cdots,y_j] \in V(\g)$, their concatenation $[x_0,x_1,\cdots,x_i,y_0,y_1,\cdots,y_j]$ may not be decreasing.
The definition of $\op{M}(\mf{x},\mf{y})$ gives several rules for converting a non-decreasing sequence to a decreasing one.

\begin{defn} [Higher multiplication $\op{M}$, special cases] \label{example}
Define $\op{M}(\mf{x},\mf{y})$, where $\mf{x},\mf{y}$ are sequences of length at most $1$ as follows:
\be
\item $\op{M}([a], [b]) = [a, b]$ if $a>b$ and $\op{M}([\es], [a]) = \op{M}([a], [\es]) = [a]$.
\item $\op{M}([a], [a]) = 0$.\footnote{Note that $0 \in \Z [q^{\pm1}, h^{\pm1}]\lan V(\g) \ran$ is different from the sequence $[0] \in V(\g)$.}
\item $\op{M}([a], [b]) = h^{(-1)^{a+b+1}}[b, a]$ if $a<b-1$.
\item $\op{M}([a], [b]) = q^{2a+1-n}[\es]+h[b,a]$ if $a=b-1$.
\ee
\end{defn}
In the rest of this subsection we generalize Definition \ref{example}.
As in Definition \ref{example} (1), the multiplication is simply gluing two sequences if the concatenation is still decreasing. We have the following definition of a gluing map $G_k$ for $k$ elements in $V(\g)$.
\begin{defn} [Gluing map $G_k$]
Let $G_k: (\Z [q^{\pm1}, h^{\pm1}]\lan V(\g) \ran)^{\times k} \rightarrow \Z [q^{\pm1}, h^{\pm1}]\lan V(\g) \ran$
be a $\Z [q^{\pm1}, h^{\pm1}]$-multilinear map defined over the basis $V(\g)^{\times k}$ as follows:
$$ G_k(\mf{x^1}, \mf{x^2}, ..., \mf{x^k})= \left\{
\begin{array}{cl} \
[x^1_0, ..., x^1_{i_1}, x^2_0, ..., x^2_{i_2}, ..., x^k_0, ..., x^k_{i_k}] & \mbox{if} \hspace{0.2cm} x^j_{i_j}>x^{j+1}_0 \hspace{0.2cm} \mbox{for} \hspace{0.2cm} 1 \leq j \leq k-1, \\
 0 & \mbox{otherwise},
\end{array}\right.
$$
where $\mf{x^1}=[x^1_0, ..., x^1_{i_1}], \mf{x^2}=[x^2_0, ..., x^2_{i_2}], ..., \mf{x^k}=[x^k_0, ..., x^k_{i_k}]$.
Here we assume $x^j_{i_j}>x^{j+1}_0$ is always true when $\mf{x^j}$ or $\mf{x^{j+1}}$ is $[\es]$.
\end{defn}

\begin{rmk}
It is easy to see that the gluing map $G_2$ is associative, i.e., $$G_2(G_2(\mf{x}, \mf{y}), \mf{z})=G_2(\mf{x}, G_2(\mf{y}, \mf{z}));$$ and that $G_k$ is a composition of $G_{k-1}$ and $G_2$: $G_k(\mf{x^1}, \mf{x^2}, ..., \mf{x^k})=G_{k-1}(G_2(\mf{x^1}, \mf{x^2}), ..., \mf{x^k})$.
\end{rmk}

The cohomological shifting number $\mu(\mf{x}, \mf{y})$ counts the shift in $h$ as in Definition \ref{example} (3) when exchanging numbers in $\mf{x}$ and $\mf{y}$ with difference greater than 1.
\begin{defn} [Cohomological shifting $\mu$]
Define $\mu: V(\g) \times V(\g) \rightarrow \Z$ by
$$\mu(\mf{x}, \mf{y})= \sum \limits_{k=0}^{i}\sum \limits_{l=0}^{j} \mu(x_k, y_l)$$
for $\mf{x}=[x_0, ..., x_i]$ and $\mf{y}=[y_0, ..., y_j]$, where
$$ \mu(x_k, y_l)= \left\{
\begin{array}{cl} \ (-1)^{x_k+y_l+1} & \mbox{if} \hspace{0.3cm} x_k < y_l-1;
\\0 & \mbox{otherwise}.
\end{array}\right.
$$
\end{defn}

Adjacent increasing pairs are special for the multiplication as in Definition \ref{example} (4).
Let $p(\mf{x}, \mf{y})$ be the number of adjacent increasing pairs $\{s, s+1~|~ s \in \mf{x}, s+1 \in \mf{y}\}$ for $\mf{x}, \mf{y} \in V(\g)$.
Since all pairs are distinct, they can be ordered as $\{s_1, s_1+1\}, ..., \{s_{p(\mf{x}, \mf{y})}, s_{p(\mf{x}, \mf{y})}+1\}$ such that
$s_1 > ... > s_{p(\mf{x}, \mf{y})}$.
Let $s_i(\mf{x}, \mf{y})$ be the smaller number in the $i$-th pair for $\mf{x}, \mf{y}$.
We write $s_i$ for $s_i(\mf{x}, \mf{y})$ when there is no confusion.
The multiplication of an adjacent pair is given in Definition \ref{example} (4).
\begin{defn} [Sequence for adjacent pairs]
Define $\beta: \{0,1,...,n-1\} \rightarrow \Z [q^{\pm1}, h^{\pm1}]\lan V(\g) \ran$ by
$$\beta(s)=q^{2s+1-n}[\es]+h[s+1,s].$$
\end{defn}

Between two adjacent pairs $\{s_i, s_i+1\}, \{s_{i+1}, s_{i+1}+1\}$, let $\alpha_i'(\mf{x}, \mf{y})$ be the $i$-th non-increasing sequence consisting of
$$\{x_k \in \mf{x}~|~ s_{i+1}+1 \leq x_k < s_i\} \hspace{0.1cm} \mbox{and} \hspace{0.1cm} \{y_l\in \mf{y}~|~ s_{i+1}+1 < y_l \leq s_i\}$$
for $0 \leq i \leq p(\mf{x}, \mf{y})$.
Here we assume $s_0=+\infty$, and $s_{p(\mf{x}, \mf{y})+1}=-\infty$.
Note that $\alpha_i'$ could be $[\es]$.
We now define the following since we want to set $\alpha_i'$ to zero if it has repetitions:

\begin{defn} [$i$-th sequence $\alpha_i$]

Define $\alpha_i: V(\g) \times V(\g) \rightarrow \Z [q^{\pm1}, h^{\pm1}]\lan V(\g) \ran$ by
$$ \alpha_i(\mf{x}, \mf{y})=D(\alpha_i'(\mf{x}, \mf{y})),$$
where $D: \{\mbox{non-increasing sequences of integers bounded by}~ n ~\mbox{and}~ 0\} \rightarrow \Z [q^{\pm1}, h^{\pm1}]\lan V(\g) \ran$ is given as
$$ D(\mf{x})= \left\{
\begin{array}{rl} \mf{x} & \mbox{if the sequence} \hspace{0.1cm} \mf{x} \hspace{0.1cm} \mbox{is decreasing};
\\0 & \mbox{otherwise}.
\end{array}\right.
$$
\end{defn}

\vspace{.2cm}
We are now in a position to define the higher multiplication $\op{M}$ in general.
\begin{defn} [Higher multiplication $\op{M}$] \label{higher}
Define $\op{M}$ over the basis $\{(\mf{x}, \mf{y}) ~|~ \mf{x}, \mf{y} \in V(\g)\}$ by
$$\op{M}(\mf{x}, \mf{y})=h^{\mu(\mf{x}, \mf{y})}  G_{2p(\mf{x}, \mf{y})+1}(\alpha_0(\mf{x}, \mf{y}), \beta(s_1(\mf{x}, \mf{y})), ..., \beta(s_{p(\mf{x}, \mf{y})}(\mf{x}, \mf{y})), \alpha_{p(\mf{x}, \mf{y})}(\mf{x}, \mf{y})).$$
\end{defn}

\begin{rmk}
(1) The definition of $\op{M}$ above reduces to Definition \ref{example} in special cases.

(2) The multiplication $\op{M}(\mf{x}, \mf{y})$ agrees with the gluing map $G_2(\mf{x}, \mf{y})$ if $G_2(\mf{x}, \mf{y}) \neq 0$.
\end{rmk}

\begin{defn} [Multiplication $\op{m}$]
The multiplication
$$\op{m}: \Z [q^{\pm1}]\lan V(\g) \ran \times \Z [q^{\pm1}]\lan V(\g) \ran \rightarrow \Z [q^{\pm1}]\lan V(\g) \ran$$
is defined as the specialization of $\op{M}$ to $h=-1$.
\end{defn}

\subsection{Associativity of the multiplication $\op{m}$}
\begin{prop}
The multiplication $\op{m}$ makes $\Z [q^{\pm1}]\lan V(\g) \ran$ into an associative $\Z [q^{\pm1}]$-algebra with unit $[\es]$.
\end{prop}

\begin{proof}
We have $\op{\op{m}}(\mf{x}, [\es])=\op{m}([\es], \mf{x})=\mf{x}$ since $\op{M}(\mf{x}, [\es])=\op{M}([\es], \mf{x})=\mf{x}$.
It remains to show that $\op{m}$ is associative, i.e.,
\begin{equation}
\op{m}(\op{m}(\mf{x}, \mf{y}), \mf{z})=\op{m}(\mf{x}, \op{m}(\mf{y}, \mf{z})) \hspace{0.2cm} \mbox{for all} \hspace{0.2cm} \mf{x}, \mf{y}, \mf{z} \in V(\g).  \tag{*}
\end{equation}
We prove (*) by induction on the length $l(\mf{y})$ of $\mf{y}$.
Assume (*) holds for $l(\mf{y}) \leq k$, where $k \geq 1$.
We show that (*) holds for $l(\mf{y}) = k+1$.
Decompose $\mf{y}=\op{m}(\mf{y_1}, \mf{y_2})$ with $l(\mf{y_1}),l(\mf{y_2}) \leq k$.
Let
$$\op{m}(\mf{x}, \mf{y_1})=\sum \limits_i q^{n_i}\mf{w}_i, \hspace{0.3cm} \op{m}(\mf{y_2}, \mf{z})=\sum \limits_j q^{n_j}\mf{v}_j,$$
where $\mf{w}_i, \mf{v}_j \in V(\g)$ and the sums are finite. Then we have
\begin{align*}
  \op{m}(\op{m}(\mf{x}, \mf{y}), \mf{z})
 = & \op{m}(\op{m}(\mf{x}, \op{m}(\mf{y_1}, \mf{y_2})), \mf{z})\\
 = & \op{m}(\op{m}(\op{m}(\mf{x}, \mf{y_1}), \mf{y_2})), \mf{z}) \hspace{0.2cm} (\mbox{since (*) holds for} \hspace{0.2cm} \mf{x}, \mf{y_1}, \mf{y_2} \hspace{0.2cm} \mbox{where} \hspace{0.2cm} l(\mf{y_1}) \leq k)\\
 = & \op{m}(\op{m}(\sum \limits_i q^{n_i}\mf{w}_i, \mf{y_2}), \mf{z}) \\
 = & \sum \limits_i q^{n_i}\op{m}(\op{m}(\mf{w}_i, \mf{y_2}), \mf{z}) \\
 = & \sum \limits_i q^{n_i}\op{m}(\mf{w}_i, \op{m}(\mf{y_2}, \mf{z})) \hspace{0.2cm} (\mbox{since (*) holds for} \hspace{0.2cm} \mf{w}_i, \mf{y_2}, \mf{z} \hspace{0.2cm} \mbox{where} \hspace{0.2cm} l(\mf{y_2}) \leq k) \\
 = & \op{m}(\sum \limits_i q^{n_i}\mf{w}_i, \op{m}(\mf{y_2}, \mf{z})) \\
= & \op{m}(\op{m}(\mf{x}, \mf{y_1}), \op{m}(\mf{y_2}, \mf{z})).
\end{align*}
Similarly, we have $\op{m}(\mf{x}, \op{m}(\mf{y}, \mf{z})) = \op{m}(\op{m}(\mf{x}, \mf{y_1}), \op{m}(\mf{y_2}, \mf{z}))$.
This proves (*) for $l(\mf{y}) = k+1$.

It remains to prove (*) for $l(\mf{y})=0, 1$.
If $l(\mf{y})=0$, i.e., $\mf{y}=[\es]$, then (*) is trivial since $[\es]$ is the unit.
The proposition follows from Lemma \ref{length} below when $l(\mf{y})=1$.
\end{proof}

\begin{lemma} \label{assocativity local}
$\op{m}$ is associative in the following cases:
\begin{align*}
(1)& \op{m}([s], \op{m}([s], [s+1]))= \op{m}(\op{m}([s], [s]), [s+1])=0 \hspace{0.1cm} for \hspace{0.1cm} 0 \leq s \leq n-1,\\
(2)& \op{m}([s], \op{m}([s+1], [s+1]))=\op{m}(\op{m}([s], [s+1]), [s+1])=0 \hspace{0.1cm} for \hspace{0.1cm} 0 \leq s \leq n-1,\\
(3)& \op{m}([s-1], \op{m}([s], [s+1]))=\op{m}(\op{m}([s-1], [s]), [s+1]) \hspace{0.1cm} for \hspace{0.1cm} 1 \leq s \leq n-1.
\end{align*}
\end{lemma}

\begin{proof}
We compute using $\op{M}$ in $\Z [q^{\pm1}, h^{\pm1}]\lan V(\g) \ran$ and then specialize to $h=-1$.

\n (1) We have $\op{M}(\op{M}([s], [s]), [s+1])=0$ since $\op{M}([s], [s])=0$. On the other hand,
\begin{align*}
\op{M}([s], \op{M}([s], [s+1]))
&= \op{M}([s], \beta(s)) \\
                     &= q^{2s+1-n}\op{M}([s], [\es]) + h~\op{M}([s], [s+1, s]) \\
&= q^{2s+1-n}[s] + h~G_2(q^{2s+1-n}[\es]+h[s+1,s], [s]) \\
                     &= q^{2s+1-n}[s] + hq^{2s+1-n}~G_2([\es], [s]) + h^2~G_2([s+1,s], [s]) \\
                     &= (1+h)q^{2s+1-n}[s].
\end{align*}
Hence $\op{m}([s], \op{m}([s], [s+1]))=\op{M}([s], \op{M}([s], [s+1]))|_{h=-1}=0$.

\n (2) Similar to that of (1).

\n (3) We have
\begin{align*}
 \op{M}([s-1], \op{M}([s], [s+1]))
=& \op{M}([s-1], q^{2s+1-n}[\es]+h[s+1,s]) \\
                     =& q^{2s+1-n}\op{M}([s-1], [\es]) + h~\op{M}([s-1], [s+1, s]) \\
                     =& q^{2s+1-n}[s-1] + h(h^{-1}G_2([s+1], \beta(s-1))) \\
                     =& q^{2s+1-n}[s-1] + G_2([s+1], q^{2(s-1)+1-n}[\es]+h[s,s-1]) \\
                     =& q^{2s+1-n}[s-1] + q^{2s-1-n}G_2([s+1], [\es]) + h~G_2([s+1] ,[s,s-1]) \\
                     =& q^{2s+1-n}[s-1] + q^{2s-1-n}[s+1] + h[s+1, s, s-1], \\
 \op{M}(\op{M}([s-1],[s]), [s+1])
=& \op{M}(q^{2(s-1)+1-n}[\es]+h[s,s-1], [s+1]) \\
                     =& q^{2s-1-n}\op{M}([\es], [s+1]) + h~\op{M}([s, s-1], [s+1]) \\
                     =& q^{2s-1-n}[s+1] + h(h^{-1}G_2(\beta(s), [s-1])) \\
                     =& q^{2s-1-n}[s+1] + G_2(q^{2s+1-n}[\es]+h[s+1,s], [s-1])
                                           \end{align*}
\begin{align*}
                     =& q^{2s-1-n}[s+1] + q^{2s+1-n}G_2([\es], [s-1]) + h~G_2([s+1,s] ,[s-1]) \\
                     =& q^{2s+1-n}[s-1] + q^{2s-1-n}[s+1] + h[s+1, s, s-1].
\end{align*}
Hence $\op{M}([s-1], \op{M}([s], [s+1]))=\op{M}(\op{M}([s-1], [s]), [s+1])$. This proves (3).
\end{proof}

\begin{rmk}
$\op{M}$ is not associative in Cases (1) and (2).
\end{rmk}

\begin{lemma}\label{length}
$\op{m}(\op{m}(\mf{x}, [t]), \mf{z})=\op{m}(\mf{x}, \op{m}([t], \mf{z})) \hspace{0.2cm} \mbox{for all} \hspace{0.2cm} \mf{x}, \mf{z} \in V(\g) \hspace{0.2cm} \mbox{and} \hspace{0.2cm} 0 \leq t \leq n.$
\end{lemma}
\begin{proof}
The proof is divided into three cases:
(A) $t-1 \notin \mf{x}$;
(B) $t+1 \notin \mf{z}$;
and (C) $t-1 \in \mf{x}, t+1 \in \mf{z}$.
Let $\mf{x} \backslash [t]= \mf{x}$ if $t \notin \mf{x}$ and let $\mf{x} \backslash [t]$ be the decreasing sequence obtained by removing $t$ from $\mf{x}$ if $t \in \mf{x}$.

\vspace{.2cm}
\n (A) Suppose $t-1 \notin \mf{x}$.
We want to express $\op{M}(\mf{x}, \op{M}([t], \mf{z}))$ and $\op{M}(\op{M}(\mf{x}, [t]), \mf{z})$ in terms of $\op{M}(\mf{x} \backslash [t], \mf{z} \backslash [t+1])$.

The number $t$ is not in any adjacent increasing pair for $(\mf{x} \backslash [t], \mf{z} \backslash [t+1])$ since $t-1,t \notin \mf{x} \backslash [t]$.
Then there exists a non-increasing sequence $\alpha_i'(\mf{x} \backslash [t], \mf{z} \backslash [t+1])=[a_0, ..., a_k]$ such that $a_0 \geq t \geq a_k$.
We decompose $\alpha_i'=\alpha_{i+}' \sqcup \alpha_{i-}'$ according to $t$, namely $\alpha_{i+}'$ is the subsequence of $\alpha_i'$ consisting of $\{a_l \in \alpha_i'~|~~ a_l > t \}$ and
$\alpha_{i-}'$ is the complementary sequence of $\alpha_{i+}'$ in $\alpha_{i}'$.
Let $\alpha_{i\pm}=D(\alpha_{i\pm}')$.
Then $\alpha_i=D(\alpha_i')=G_2(\alpha_{i+}, \alpha_{i-})$ and
\begin{equation*}
\begin{aligned}
\op{M}(\mf{x} \backslash [t], \mf{z} \backslash [t+1]) &= h^{\mu_1}G_{2p+1}(\alpha_0, \beta_1, ...,\alpha_i,..., \alpha_p)\\
&= h^{\mu_1}G_{2p+1}(\alpha_0, \beta_1, ...,G_2(\alpha_{i+}, \alpha_{i-}),..., \alpha_p)\\
&= h^{\mu_1}G_{2p+2}(\alpha_0, \beta_1, ...,\alpha_{i+}, \alpha_{i-},..., \alpha_p),
\end{aligned}
\end{equation*}
where $\mu_1=\mu(\mf{x} \backslash [t], \mf{z} \backslash [t+1])$.
Then we have
\begin{equation*}
\begin{aligned}
\op{M}(\mf{x}, \op{M}([t], \mf{z}))&=h^{\mu_2}h^{\mu_1}G_{2p+3}(\alpha_0, \beta_1, ...,\alpha_{i+}, \op{M}(\mf{a}, \op{M}([t], \mf{b})), \alpha_{i-},..., \alpha_p), \\
\op{M}(\op{M}(\mf{x}, [t]), \mf{z})&=h^{\mu_2}h^{\mu_1}G_{2p+3}(\alpha_0, \beta_1, ...,\alpha_{i+}, \op{M}(\op{M}(\mf{a},[t]), \mf{b}), \alpha_{i-},..., \alpha_p),
\end{aligned}
\end{equation*}
where $\mu_2=\mu(\mf{x}, [t])+\mu([t], \mf{z})$, and
$$ (\mf{a},\mf{b})= \left\{
\begin{array}{cl}
([t], [t+1]) & \mbox{if}~~ t\in \mf{x}, t+1\in \mf{z}; \\
([t], [\es]) & \mbox{if}~~ t\in \mf{x}, t+1\notin \mf{z}; \\
([\es], [t+1]) & \mbox{if}~~ t\notin \mf{x}, t+1\in \mf{z}; \\
([\es], [\es]) & \mbox{if}~~ t\notin \mf{x}, t+1\notin \mf{z}.
\end{array}\right.
$$
Since $\op{m}(\mf{a}, \op{m}([t], \mf{b}))=\op{m}(\op{m}(\mf{a},[t]), \mf{b})$ in all the cases by Lemma \ref{assocativity local}, we obtain $\op{m}(\op{m}(\mf{x}, [t]), \mf{z})=\op{m}(\mf{x}, \op{m}([t], \mf{z}))$.

\vspace{0.2cm}
\n (B) Similar to that of (A).

\vspace{0.2cm}
\n (C) Suppose $t-1 \in \mf{x}, t+1 \in \mf{z}$.
We use the above method for $\mf{x}\backslash[t-1], \mf{z} \backslash [t+1]$:
$$\op{M}(\mf{x}\backslash[t-1], \mf{z} \backslash [t+1]) = h^{\mu_1}G_{2p+2}(\alpha_0, \beta_1, ...,\alpha_{i+}, \alpha_{i-},..., \alpha_p).$$
where $\mu_1=\mu(\mf{x}\backslash[t-1], \mf{z} \backslash [t+1])$.
Then we have
\begin{equation*}
\begin{aligned}
\op{M}(\mf{x}, \op{M}([t], \mf{z}))&=h^{\mu_2}h^{\mu_1}G_{2p+3}(\alpha_0, \beta_1, ...,\alpha_{i+}, \op{M}([t-1], \op{M}([t], [t+1])), \alpha_{i-},..., \alpha_p), \\
\op{M}(\op{M}(\mf{x}, [t]), \mf{z})&=h^{\mu_2}h^{\mu_1}G_{2p+3}(\alpha_0, \beta_1, ...,\alpha_{i+}, \op{M}(\op{M}([t-1],[t]), [t+1]), \alpha_{i-},..., \alpha_p),
\end{aligned}
\end{equation*}
where $\mu_2=\mu(\mf{x}, [t])+\mu([t], \mf{z})$.
Since $\op{m}([t-1], \op{m}([t], [t+1]))=\op{m}(\op{m}([t-1], [t]), [t+1])$ by Lemma \ref{assocativity local}, we obtain $\op{m}(\op{m}(\mf{x}, [t]), \mf{z})=\op{m}(\mf{x}, \op{m}([t], \mf{z}))$.
\end{proof}

Now we are ready to describe the algebra structure on $K_0(R)=\Z [q^{\pm1}]\lan V(\g) \ran$. It is easy to see that $K_0(R)$ is generated by sequences of length $1$ since any sequence of length $k$ is a product of $k$ sequences of length $1$. Relations in $K_0(R)$ come from $\op{m}$ for the special cases as in Definition \ref{example}.

\begin{prop} \label{K_0}
Let $X_{\es}$ denote the vertex $[\es] \in V(\gn)$ and $X_i$ denote the vertex $[i] \in V(\gn)$ for $0 \leq i \leq n$.
Then $K_0(R)$ is an associative $\Z[q^{\pm1}]$-algebra with the unit $X_{\es}$, generators $X_i$'s and relations:
\begin{gather*}
X_i^2=0;\\
X_iX_j=-X_jX_i \hspace{.2cm} \mbox{if} \hspace{.2cm} |i-j|>1;\\
X_iX_{i+1}+X_{i+1}X_i=q^{2i+1-n}.
\end{gather*}
\end{prop}

\section{Categorification of the multiplication on $K_0(\rn)$}
We define a functor $\cal{M}_n: DGP(\drnn) \ra DGP(\rn)$ by tensoring with a $q$-graded DG $(\rn, \drnn)$-bimodule $T_n$.
We prove that $\cal{M}_n$ preserves closed and exact morphisms so that it induces an exact functor $\cal{M}_n|_{H^0}: H^0(DGP(\drnn)) \ra H^0(DGP(\rn))$.
Finally, we show that $\cal{M}_n|_{H^0}$ categorifies the multiplication $\op{m}_n$.

\subsection{The $q$-graded DG $(R, \drr)$-bimodule $T$}
In this section we construct a $q$-graded DG bimodule $T=\bigoplus T^k$, where $T^k$ is the summand of cohomological degree $k$.
The differential $d=\sum d^k$, $d^k: T^k \ra T^{k+1}$, satisfies
$$d(a\cdot m)=a\cdot d(m), \hspace{.3cm} d(m\cdot r)=d(m)\cdot r+m\cdot d(r),$$
for $m \in T, a \in R, r \in \drr$.

The DG $(R, \drr)$-bimodule $T$ is constructed in the following steps:
\be
\item We define the left $R$-module $T$ in Section 4.1.1;
\item We then define the differential $d$ of $T$ as a left DG $R$-module in Section 4.1.2;
\item We define the right DG $(\drr)$-module structure of $T$ in Sections 4.1.3 and 4.1.4.
\ee

\subsubsection{$T$ as a left $R$-module}
Since $T\otimes_{\drr}P(\mf{x,y})$ is expected to be a DG $R$-module in $DGP(R)$ which descends to $\op{m}(\mf{x}, \mf{y})$ in the Grothendieck group $K_0(R)$ of $H^0(DGP(R))$, we use the higher multiplication $\op{M}$ to construct $T$.

Recall from Definition \ref{higher} that:
$$\op{M}(\mf{x}, \mf{y})=h^{\mu(\mf{x}, \mf{y})}  G_{2p(\mf{x}, \mf{y})+1}(\alpha_0, \beta(s_1(\mf{x}, \mf{y})), ..., \beta(s_{p(\mf{x}, \mf{y})}(\mf{x}, \mf{y})), \alpha_{p(\mf{x}, \mf{y})}).$$
Note that $\op{M}(\mf{x}, \mf{y})$ is of degree $\mu(\mf{x}, \mf{y})+p(\mf{x}, \mf{y})$ as a Laurent polynomial of $h$.

Let $\op{M}^k: V(\g) \times V(\g) \rightarrow \Z [q^{\pm1}]\lan V(\g) \ran$ be the coefficient of $h^k$ in $\op{M}$:
$$\op{M}(\mf{x}, \mf{y})=\sum \limits_{k=-\infty}^{+\infty} \op{M}^k(\mf{x}, \mf{y})h^k
=\sum \limits_{k=\mu(\mf{x}, \mf{y})}^{\mu(\mf{x}, \mf{y})+p(\mf{x}, \mf{y})} \op{M}^k(\mf{x}, \mf{y})h^k.$$
We want to expand $\op{M}^k(\mf{x}, \mf{y})$ further in terms of $q$.
We will omit $\mf{x}, \mf{y}$ in $\mu(\mf{x}, \mf{y}), p(\mf{x}, \mf{y}), s_i(\mf{x}, \mf{y})$ when $\mf{x}, \mf{y}$ are understood.

Let $\cal{I}_{k-\mu}(\mf{x}, \mf{y})= \{A \subset \{1,2,...,p\}~|~ |A|=k-\mu\}$ be the collection of all $(k-\mu)$-element subsets of $\{1,2,...,p\}$ for $\mu \leq k \leq \mu+p$.
Let $\eta: \cal{I}_{k-\mu}(\mf{x}, \mf{y}) \ra \Z$ be the overall shift in $q$ defined by
$$\eta(A)= \sum \limits_{i \notin A}(2s_i+1-n).$$
Let $\beta_{A}: \{s_1, ..., s_p\} \ra V(\g)$ be a choice of components in $\beta$ for an index set $A \in \cal{I}_{k-\mu}(\mf{x}, \mf{y})$:
$$ \beta_{A}(s_i)= \left\{
\begin{array}{cl} \
[s_i +1, s_i] & \mbox{if} \hspace{0.3cm} i \in A; \\

[\es] & \mbox{otherwise}.
\end{array}\right.
$$

\begin{defn} [Expansion of $\op{M}^k$]
For any given $\mf{x}, \mf{y} \in V(\g)$ and $A \in \cal{I}_{k-\mu}(\mf{x}, \mf{y})$,
let
$$\op{M}_A ^k(\mf{x}, \mf{y})= G_{2p+1}(\alpha_0, \beta_A(s_1), ..., \beta_A(s_p), \alpha_p) \in V(\g) \sqcup \{0\}$$
be the coefficient of $q^{\eta(A)}$ in $\op{M}^k(\mf{x}, \mf{y})$:
$$\op{M}^k(\mf{x}, \mf{y})= \sum \limits_{A \in \cal{I}_{k-\mu}} \op{M}_A ^k(\mf{x}, \mf{y})q^{\eta(A)}.$$
\end{defn}

Recall that $P(\mf{x})=R\cdot e(\mf{x})$ and $P(\mf{x})\{n\}$ is $P(\mf{x})$ with the $q$-grading shifted by $n$.
Let $P(0)=0$ denote the trivial $R$-module.
We are now in a position to define $T^k$ as a left $R$-module.

\begin{defn} [$R$-module $T^k$]
Define $T^k = \bigoplus \limits_{\mf{x}, \mf{y} \in V(\g)} T^k(\mf{x}, \mf{y})$ as left projective $R$-modules, where
$$T^k(\mf{x}, \mf{y}) = \bigoplus_{A \in \cal{I}_{k-\mu}(\mf{x}, \mf{y})} P(\op{M}_A ^k(\mf{x}, \mf{y}))\{\eta(A)\}.$$
\end{defn}

\begin{rmk} \label{class}
The class of $T^k(\mf{x}, \mf{y})$ in the Grothendieck group $K_0(R)$ is $\op{M}^k(\mf{x}, \mf{y})$.
\end{rmk}

\subsubsection{$T$ as a left DG $R$-module}
We define $d^k: T^k \ra T^{k+1}$ as a left $R$-module differential by
$$d^k = \sum \limits_{\mf{x}, \mf{y} \in V(\g)} d^k(\mf{x}, \mf{y}) = \sum \limits_{\mf{x}, \mf{y} \in V(\g)} \sum_{\tiny{\begin{array}{c}A \in \cal{I}_{k-\mu}\\B \in \cal{I}_{k+1-\mu}\end{array}}}d_A^B(\mf{x}, \mf{y}),$$
where $d^k(\mf{x}, \mf{y}): T^k(\mf{x}, \mf{y}) \to T^{k+1}(\mf{x}, \mf{y})$ is defined on each summand by $$d_A^B(\mf{x}, \mf{y}): P(\op{M}_A ^k(\mf{x}, \mf{y}))\{\eta(A)\} \to P(\op{M}_B ^{k+1}(\mf{x}, \mf{y}))\{\eta(B)\}.$$

For any given $\mf{x}, \mf{y} \in V(\g)$, if $A \in \cal{I}_{k-\mu}$ is a subset of $B \in \cal{I}_{k+1-\mu}$ and
\begin{gather*}
\op{M}_A ^k=G_{2p+1}(\alpha_0, \beta_A(s_1), ..., \beta_A(s_p), \alpha_p) \\
\op{M}_B ^{k+1}=G_{2p+1}(\alpha_0, \beta_B(s_1), ..., \beta_B(s_p), \alpha_p)),
\end{gather*}
are both nonzero,
then they only differ by a pair of adjacent numbers $\{s_{i(A,B)}+1, s_{i(A,B)}\}$ at $\beta_A(s_{i(A,B)})$ and $\beta_B(s_{i(A,B)})$.
Here we write $i(A,B)$ for the unique element in $A-B$.
Hence there exists a generator $r(\op{M}_A ^k \xra{s_{i(A,B)}} \op{M}_B ^{k+1}) \in R$.

\begin{defn}
For any given $\mf{x}, \mf{y} \in V(\g)$, $A \in \cal{I}_{k-\mu}, B \in \cal{I}_{k+1-\mu}$, let
$$d_A^B(\mf{x}, \mf{y}): P(\op{M}_A ^k)\{\eta(A)\} \ra P(\op{M}_B ^{k+1})\{\eta(B)\}$$
be the left $R$-module map given by multiplying $r(\op{M}_A ^k \xra{s_{i(A,B)}} \op{M}_B ^{k+1})$ from the right if $B = A \sqcup \{i(A,B)\}$ and $\op{M}_A ^k, \op{M}_B ^{k+1} \in V(\g)$.
Otherwise let $d_A^B(\mf{x}, \mf{y})=0$.
\end{defn}

\begin{rmk}
The map $d_A^B$ preserves the $q$-grading because of the $q$ shifting $\{\eta(A)\}, \{\eta(B)\}$ on the modules.
\end{rmk}

\begin{lemma}
$d$ is a differential, i.e., $d^{k+1} \circ d^{k}=0.$
\end{lemma}

\begin{proof}
It suffices to prove that
\begin{equation}
d^{k+1}(\mf{x}, \mf{y})\circ d^{k}(\mf{x}, \mf{y}){\big |}_{P(\op{M}_A ^k)\{\eta(A)\}}=\sum_{\tiny{\begin{array}{c}B \in \cal{I}_{k+1-\mu}\\C \in \cal{I}_{k+2-\mu}\end{array}}}d_B^C \circ d_A^B=0 \tag{**}
\end{equation}
for any $A \in \cal{I}_{k-\mu}(\mf{x}, \mf{y})$ and any pair $\mf{x}, \mf{y} \in V(\g)$.
By definition $d_A^B, d_B^C$ are both nonzero if and only if
$$B = A \sqcup \{i(A,B)\},~~~~~~ C = B \sqcup \{i(B,C)\},$$
for some $i(A,B), i(B,C)$ and $\op{M}_A ^k, \op{M}_B ^{k+1}, \op{M}_C ^{k+2} \in V(\g)$.
Then there exists another index set $B' = A \sqcup \{i(B,C)\}$ such that $C = B' \sqcup \{i(A,B)\}$ and $\op{M}_{B'} ^{k+1} \in V(\g)$.
Hence $d_A^{B'}, d_{B'}^C$ are both nonzero by definition.
The map
$$d_B^C \circ d_A^B: P(\op{M}_A ^k)\{\eta(A)\} \ra P(\op{M}_C ^{k+2})\{\eta(C)\}$$
is the right multiplication by
$$r(\op{M}_A ^k \xra{s_{i(A,B)}} \op{M}_B ^{k+1}) \cdot r(\op{M}_B ^{k+1} \xra{s_{i(B,C)}} \op{M}_C ^{k+2})$$
and the map $d_{B'}^C \circ d_A^{B'}$ is the right multiplication by
$$r(\op{M}_A ^k \xra{s_{i(B,C)}} \op{M}_{B'} ^{k+1}) \cdot r(\op{M}_{B'} ^{k+1} \xra{s_{i(A,B)}} \op{M}_C ^{k+2}).$$
Hence $d_B^C \circ d_A^B + d_{B'}^C \circ d_A^{B'}=0$ since
$$r(\op{M}_A ^k \xra{s_{i(A,B)}} \op{M}_B ^{k+1}) \cdot r(\op{M}_B ^{k+1} \xra{s_{i(B,C)}} \op{M}_C ^{k+2})=r(\op{M}_A ^k \xra{s_{i(B,C)}} \op{M}_{B'} ^{k+1}) \cdot r(\op{M}_{B'} ^{k+1} \xra{s_{i(A,B)}} \op{M}_C ^{k+2}).$$
This implies Equation (**).
\end{proof}

\subsubsection{The right $\drr$-multiplication.}
In this subsection we define the right multiplication with the generators $e(\mf{x},\mf{y})$ and $r(\mf{x},\mf{y} \xra{s,t} \mf{x'},\mf{y'})$ of $\drr$.
Let $m \times r$ denote the right multiplication for $m \in T, r \in \drr$ and $m \cdot a$ denote the multiplication in $R$ for $m \in P(\mf{x}) \subset R, a \in R$.
The definition of $m\times r$ is rather involved and occupies the next several pages.

\vspace{.2cm}
\n (1) Let $r=e(\mf{x},\mf{y})$. Then define
$m \times e(\mf{x},\mf{y})=\delta_{\mf{x},\mf{x'}} \delta_{\mf{y},\mf{y'}} m$ for $m \in T(\mf{x'},\mf{y'})$.

\vspace{.2cm}
\n (2) Let $r=r(\mf{x},\mf{y} \xra{\es,t} \mf{x},\mf{y'})$.
 Let us abbreviate $\mu=\mu(\mf{x}, \mf{y})$, $\mu'=\mu(\mf{x}, \mf{y'})$, $p=p(\mf{x}, \mf{y})$, $p'=p(\mf{x}, \mf{y'})$, $s_i=s_i(\mf{x}, \mf{y})$ and $s_i'=s_i(\mf{x}, \mf{y'})$.
Since $\mf{y'}=\mf{y}\sqcup\{t+1,t\}$ and in particular $t+1 \notin \mf{y}$, we have $t \notin \{s_1, ..., s_{p}\}$. Let $a(t) \in \{1, ..., p\}$ be the number such that $s_{a(t)} > t > s_{a(t)+1}$.

The right multiplication
$$
\begin{array}{cccc}
\times r(\mf{x},\mf{y} \xra{\es,t} \mf{x},\mf{y'}):
& \bigoplus\limits_{A \in \cal{I}_{k-\mu}} P(\op{M}_A ^k(\mf{x}, \mf{y}))\{\eta(A)\}& \ra & \bigoplus\limits_{B \in \cal{I}_{k-\mu'}} P(\op{M}_B ^{k}(\mf{x}, \mf{y'}))\{\eta(B)\}
\end{array}
$$
is defined on a case-by-case basis as follows:

\vspace{.2cm}
\n (2A) Suppose $t-1 \notin \mf{x}, t \notin \mf{x}.$
We have $\mu'=\mu, p'=p$.
We decompose $\alpha_{a(t)}=G_2(\alpha_{a(t)+}, \alpha_{a(t)-}),$ where $\alpha_{a(t)+}$ is the subsequence of $\alpha_{a(t)}$ consisting of numbers greater than $t$ and $\alpha_{a(t)-}$ is the complementary sequence of $\alpha_{a(t)+}$ in $\alpha_{a(t)}$.
Then we have
    \begin{equation*}
    \begin{aligned}
     \op{M}(\mf{x}, \mf{y})=
    & h^{\mu}  G_{2p+1}(\alpha_0, ..., \alpha_{a(t)}, ..., \alpha_{p}) \\
    = & h^{\mu}  G_{2p+2}(\alpha_0, ..., \alpha_{a(t)+}, \alpha_{a(t)-}, ..., \alpha_{p}), \\
    \op{M}(\mf{x}, \mf{y'})=
    & h^{\mu}  G_{2p+3}(\alpha_0, ..., \alpha_{a(t)+}, [t+1, t], \alpha_{a(t)-}, ..., \alpha_{p}).
    \end{aligned}
    \end{equation*}
    Define the function $f: \cal{I}_{k-\mu}(\mf{x}, \mf{y}) \ra \cal{I}_{k-\mu'}(\mf{x}, \mf{y'})$ as the identity. Then two non-increasing sequences $\op{M}_A ^k(\mf{x}, \mf{y})$ and $\op{M}_{f(A)} ^k(\mf{x}, \mf{y'})$ differ by the pair $\{t+1, t\}$.
    If the two sequences are decreasing, then there exists a generator $r(\op{M}_A ^k(\mf{x}, \mf{y}) \xra{t} \op{M}_{f(A)} ^k(\mf{x}, \mf{y'})) \in R$;
    otherwise we write $r(\op{M}_A ^k(\mf{x}, \mf{y})$ $\xra{t} \op{M}_{f(A)} ^k(\mf{x}, \mf{y'}))$ to denote $0$.

    If $m \in P(\op{M}_A ^k(\mf{x}, \mf{y}))\{\eta(A)\}$, then we define
    $$m \times r(\mf{x},\mf{y} \xra{\es,t} \mf{x},\mf{y'})=
    m \cdot r(\op{M}_A ^k(\mf{x}, \mf{y}) \xra{t} \op{M}_{f(A)} ^k(\mf{x}, \mf{y'})) \in P(\op{M}_{f(A)} ^k(\mf{x}, \mf{y'}))\{\eta(f(A))\}.$$

\vspace{.1cm}
\n (2B) Suppose $t-1 \notin \mf{x}, t \in \mf{x}.$
We have $\mu'=\mu, p'=p+1$ and the same decomposition $\alpha_{a(t)}=G_2(\alpha_{a(t)+}, \alpha_{a(t)-})$ as that in (2A). In particular, $t \in \alpha_{a(t)-}$.
The number $t$ is in some increasing adjacent pair for $\mf{x}, \mf{y'}$ since $t \in \mf{x}$ and $t+1 \in \mf{y'}$.
More precisely, $t=s_{a(t)+1}' \in \{s_1', ..., s_{p'}'\}$.
    \begin{align*}
    \op{M}(\mf{x}, \mf{y})&=
     h^{\mu}  G_{2p+2}(\alpha_0, ...,\beta(s_{a(t)}), \alpha_{a(t)+}, \alpha_{a(t)-}, ..., \alpha_{p}), \\
    \op{M}(\mf{x}, \mf{y'})&=
    h^{\mu}  G_{2p+3}(\alpha_0, ..., \beta(s_{a(t)}'), \alpha_{a(t)+}, \beta(t), \alpha_{a(t)-}, ..., \alpha_{p})\\
    &= h^{\mu}  G_{2p+3}(\alpha_0, ..., \beta(s_{a(t)}'), \alpha_{a(t)+}, \beta(s_{a(t)+1}'), \alpha_{a(t)-}, ..., \alpha_{p}).
    \end{align*}
    Define $f: \cal{I}_{k-\mu}(\mf{x}, \mf{y}) \ra \cal{I}_{k-\mu'}(\mf{x}, \mf{y'})$ for $A \in \cal{I}_{k-\mu}(\mf{x}, \mf{y})$ by
    $$f(A)=\{a~|~ a \in A, a \leq a(t)\} \sqcup \{a+1~|~ a \in A, a>a(t)\}.$$
    \noindent We have $\beta_{f(A)}(s_{a(t)+1}')=[\es]$ since $a(t)+1 \notin f(A)$.
    Hence $\op{M}_A ^k(\mf{x}, \mf{y})=\op{M}_{f(A)} ^k(\mf{x}, \mf{y'})$.

    If $m \in P(\op{M}_A ^k(\mf{x}, \mf{y}))\{\eta(A)\}$, then we define
    $$m \times r(\mf{x},\mf{y} \xra{\es,t} \mf{x},\mf{y'})=
    m \cdot e(\op{M}_{f(A)} ^k(\mf{x}, \mf{y'}))\in P(\op{M}_{f(A)} ^k(\mf{x}, \mf{y'}))\{\eta(f(A))\}.$$

\vspace{.1cm}
\n (2C) Suppose $t-1 \in \mf{x}, t \notin \mf{x}.$
We have $\mu'=\mu+\mu(t-1, t+1)=\mu-1$ and $p'=p+1$.
We decompose $$\alpha_{a(t)}=G_3(\alpha_{a(t)+}, [t-1], \alpha_{a(t)-}),$$ where $\alpha_{a(t)+}$ is the subsequence of $\alpha_{a(t)}$ consisting of numbers greater than $t$ and $\alpha_{a(t)-}$ is the subsequence of $\alpha_{a(t)}$ consisting of numbers less than $t-1$.
The number $t-1$ is in some increasing adjacent pair for $\mf{x}, \mf{y'}$ since $t-1 \in \mf{x}$ and $t \in \mf{y'}$.
More precisely, $t-1=s_{a(t)+1}'$.
    \begin{equation*}
    \begin{aligned}
    \op{M}(\mf{x}, \mf{y})=
    & h^{\mu}  G_{2p+3}(\alpha_0, ..., \beta(s_{a(t)}), \alpha_{a(t)+}, [t-1], \alpha_{a(t)-}, ..., \alpha_{p}),\\
    \op{M}(\mf{x}, \mf{y'})=
    & h^{\mu'}  G_{2p+4}(\alpha_0, ..., \beta(s_{a(t)}'), \alpha_{a(t)+}, [t+1], \beta(t-1), \alpha_{a(t)-}, ..., \alpha_{p})\\
    =
    & h^{\mu'}  G_{2p+4}(\alpha_0, ..., \beta(s_{a(t)}'), \alpha_{a(t)+}, [t+1], \beta(s_{a(t)+1}'), \alpha_{a(t)-}, ..., \alpha_{p}).
    \end{aligned}
    \end{equation*}
    Define $f: \cal{I}_{k-\mu}(\mf{x}, \mf{y}) \ra \cal{I}_{k-\mu'}(\mf{x}, \mf{y'})$ for $A \in \cal{I}_{k-\mu}(\mf{x}, \mf{y})$ by
    $$f(A)=\{a~|~ a \in A, a \leq a(t)\} \sqcup \{a(t)+1\} \sqcup \{a+1~|~ a \in A, a>a(t)\}.$$
    We have $\beta_{f(A)}(s_{a(t)+1}')=[t, t-1]$ since $a(t)+1 \in f(A)$.
    Then there exists a generator $r(\op{M}_A ^k(\mf{x}, \mf{y}) \xrightarrow{t} \op{M}_{f(A)} ^k(\mf{x}, \mf{y'}))$ if $\op{M}_A ^k(\mf{x}, \mf{y})$ and $\op{M}_{f(A)} ^k(\mf{x}, \mf{y'})$ are both nonzero.
    The definition of the right multiplication is the same as that in (2A).

\vspace{.1cm}
\n (2D) Suppose $t-1, t \in \mf{x}.$
We have $\mu'=\mu-1$ and $p'=p+2$.
We decompose $$\alpha_{a(t)}=G_4(\alpha_{a(t)+}, [t], [t-1], \alpha_{a(t)-}),$$ where $\alpha_{a(t)+}$ is the subsequence of $\alpha_{a(t)}$ consisting of numbers greater than $t$ and $\alpha_{a(t)-}$ is the subsequence of $\alpha_{a(t)}$ consisting of numbers less than $t-1$.
The numbers $t$ and $t-1$ are in some increasing adjacent pairs for $\mf{x}, \mf{y'}$ since $t-1,t \in \mf{x}$ and $t,t+1 \in \mf{y'}$.
More precisely, we have $t=s_{a(t)+1}'$ and $t-1=s_{a(t)+2}'$.
    \begin{equation*}
    \begin{aligned}
    \op{M}(\mf{x}, \mf{y})=
    & h^{\mu}  G_{2p+4}(\alpha_0, ..., \beta(s_{a(t)}), \alpha_{a(t)+}, [t], [t-1], \alpha_{a(t)-}, ..., \alpha_{p});\\
    \op{M}(\mf{x}, \mf{y'})=
    & h^{\mu'}  G_{2p+5}(\alpha_0, ..., \beta(s_{a(t)}'), \alpha_{a(t)+}, \beta(t), [\es], \beta(t-1), \alpha_{a(t)-}, ..., \alpha_{p})\\
    =
    & h^{\mu'}  G_{2p+5}(\alpha_0, ..., \beta(s_{a(t)}'), \alpha_{a(t)+}, \beta(s_{a(t)+1}'), [\es], \beta(s_{a(t)+2}'), \alpha_{a(t)-}, ..., \alpha_{p}).
    \end{aligned}
    \end{equation*}
    Define $f: \cal{I}_{k-\mu}(\mf{x}, \mf{y}) \ra \cal{I}_{k-\mu'}(\mf{x}, \mf{y'})$ for $A \in \cal{I}_{k-\mu}(\mf{x}, \mf{y})$ by
    $$f(A)=\{a~|~ a \in A, a \leq a(t)\} \sqcup \{a(t)+2\} \sqcup \{a+2~|~ a \in A, a>a(t)\}.$$
   We have $\beta_{f(A)}(s_{a(t)+1}')=[\es], \beta_{f(A)}(s_{a(t)+2}')=[t, t-1]$ since $a(t)+1 \notin f(A), a(t)+2 \in f(A)$. Hence $\op{M}_A ^k(\mf{x}, \mf{y})=\op{M}_{f(A)} ^k(\mf{x}, \mf{y'})$. The definition of the right multiplication is the same as that in (2B).

\vspace{.2cm}
\n (3)
Let $r=r(\mf{x},\mf{y} \xra{t,\es} \mf{x'},\mf{y})$.
This is similar to (2) and definition of the right multiplication breaks into 4 cases, depending on whether $t+2\in X$ and whether $t+1\in X$.

\vspace{.2cm}
\n (4)
Let $r=r(\mf{x},\mf{y} \xra{t,t+1} \mf{x'},\mf{y'})$.
Let us abbreviate $\mu'=\mu(\mf{x'}, \mf{y'})$, $p'=p(\mf{x'}, \mf{y'})$ and $s_i'=s_i(\mf{x'}, \mf{y'})$.
We have $\mu'=\mu-1, p'=p+2$ and the same decomposition $\alpha_{a(t)}=G_2(\alpha_{a(t)+}, \alpha_{a(t)-})$ as that in (2A).
The numbers $t+1$ and $t$ are in some increasing adjacent pairs for $\mf{x'}, \mf{y'}$ since $t,t+1 \in \mf{x}$ and $t+1,t+2 \in \mf{y'}$.
More precisely, we have $t+1=s_{a(t)+1}'$ and $t=s_{a(t)+2}'$.
\begin{equation*}
    \begin{aligned}
    \op{M}(\mf{x}, \mf{y})=
    & h^{\mu}  G_{2p+2}(\alpha_0, ..., \beta(s_{a(t)}), \alpha_{a(t)+}, \alpha_{a(t)-}, ..., \alpha_{p});\\
    \op{M}(\mf{x'}, \mf{y'})=
    & h^{\mu'}  G_{2p+5}(\alpha_0, ..., \beta(s_{a(t)}'), \alpha_{a(t)+}, \beta(t+1), [\es], \beta(t), \alpha_{a(t)-}, ..., \alpha_{p})\\
    =
    & h^{\mu'}  G_{2p+5}(\alpha_0, ..., \beta(s_{a(t)}'), \alpha_{a(t)+}, \beta(s_{a(t)+1}'), [\es], \beta(s_{a(t)+2}'), \alpha_{a(t)-}, ..., \alpha_{p}).
    \end{aligned}
    \end{equation*}
    Define $f: \cal{I}_{k-\mu}(\mf{x}, \mf{y}) \ra \cal{I}_{k-1-\mu'}(\mf{x'}, \mf{y'})$ for $A \in \cal{I}_{k-\mu}(\mf{x}, \mf{y})$ by
    $$f(A)=\{a~|~ a \in A, a \leq a(t)\} \sqcup \{a+2~|~ a \in A, a>a(t)\}.$$
    \noindent We have $\beta_{f(A)}(s_{a(t)+1}')=\beta_{f(A)}(s_{a(t)+2}')=[\es]$ since $a(t)+1, a(t)+2 \notin f(A)$. Hence $\op{M}_A ^k(\mf{x}, \mf{y})=\op{M}_{f(A)} ^{k-1}(\mf{x'}, \mf{y'})$.
Finally, if $m \in P(\op{M}_A ^k(\mf{x}, \mf{y}))\{\eta(A)\}$, then we define
    $$m \times r(\mf{x},\mf{y} \xra{t,t+1} \mf{x'},\mf{y'})=
    m \cdot e(\op{M}_{f(A)} ^{k-1}(\mf{x'}, \mf{y'}))\in P(\op{M}_{f(A)} ^{k-1}(\mf{x'}, \mf{y'}))\{\eta(f(A))\}.$$

\vspace{0.2cm}
\n This concludes the definition of the right $\drr$-multiplication.

\begin{rmk}
The definition above is compatible with the $q$-grading on $T$.
\end{rmk}

We need to show that the above definition gives $T$ a right DG $\drr$-module structure.
More precisely, we need to verify that
\be
\item $(m\times r_1)\times r_2 = (m\times r_1') \times r_2',$
if $r_1 \cdot r_2 = r_1' \cdot r_2'$ for $m \in T$ and generators $r_1, r_2, r_1', r_2' \in \drr$.

\item $d(m \times r)=d(m) \times r + m \times d(r),$
for $m \in T$ and $r \in \drr$.
\ee
We prove the second equation for $r=r(\mf{x},\mf{y} \xra{t,t+1} \mf{x'},\mf{y'})$ in the next subsection and leave others for the reader.
\subsubsection{$T$ as a right DG $(\drr)$-module}
\begin{lemma} The differential satisfies the Leibniz rule with respect to the right multiplication:
$$d(m \times r)=d(m) \times r + m \times d(r),$$
for $m \in P(\op{M}_A ^k(\mf{x}, \mf{y}))\{\eta(A)\}$ and $r=r(\mf{x},\mf{y} \xra{t,t+1} \mf{x'},\mf{y'})$.
\end{lemma}
\proof
Consider the index maps in the definition of right multiplication with $r(\mf{x},\mf{y} \xra{t,t+1} \mf{x'},\mf{y'})$:
\begin{align*}
f: \cal{I}_{k-\mu}(\mf{x}, \mf{y})  \ra \cal{I}_{k-1-\mu'}(\mf{x'}, \mf{y'}),\\
g: \cal{I}_{k+1-\mu}(\mf{x}, \mf{y})  \ra \cal{I}_{k-\mu'}(\mf{x'}, \mf{y'}),
\end{align*}
\begin{gather*}
f(A)=\{a~|~ a \in A, a \leq a(t)\} \sqcup \{a+2~|~ a \in A, a>a(t)\}, \\ g(B)=\{b~|~ b \in B, b \leq a(t)\} \sqcup \{b+2~|~ b \in B, b>a(t)\},
\end{gather*}
for $A \in \cal{I}_{k-\mu}(\mf{x}, \mf{y})$ and $B \in \cal{I}_{k+1-\mu}(\mf{x}, \mf{y})$.

Consider $\overline{\cal{B'}}=\{B' \in \cal{I}_{k+1-\mu}(\mf{x}, \mf{y}) ~|~ B' \supset A\}$ and
$$\begin{array}{rl}
\cal{B}= & \{B \in \cal{I}_{k-\mu'}(\mf{x'}, \mf{y'}) ~|~ B=f(A) \sqcup \{i(f(A), B)\}\} \\
= &  \{B_1=f(A)\sqcup\{a(t)+1\}\} \sqcup \{B_2=f(A)\sqcup\{a(t)+2\}\} \\
  & \qquad \qquad  \sqcup \{B ~|~ B=g(B')~\mbox{for some}~ B' \in \overline{\cal{B'}}\} \\
= &  \{B_1\} \sqcup \{B_2\} \sqcup \overline{\cal{B}}.
\end{array}
$$
Then we have
\begin{align*}
   d\left(m \times r(\mf{x},\mf{y} \xra{t,t+1} \mf{x'},\mf{y'})\right)
= & d(m \cdot e(\op{M}_{f(A)} ^{k-1}(\mf{x'}, \mf{y'}))) \\
= & \sum \limits_{B \in \cal{B}} m \cdot r(\op{M}_{f(A)} ^{k-1}(\mf{x'}, \mf{y'}) \xra{s_{i(f(A),B)}}  \op{M}_{B} ^{k}(\mf{x'}, \mf{y'})) \\
= & \sum \limits_{B \in \overline{\cal{B}}} m \cdot r(\op{M}_{f(A)} ^{k-1}(\mf{x'}, \mf{y'}) \xra{s_{i(f(A),B)}}  \op{M}_{B} ^{k}(\mf{x'}, \mf{y'})) \\
  &  + m \cdot r(\op{M}_{f(A)} ^{k-1}(\mf{x'}, \mf{y'}) \xra{s_{a(t)+1}}  \op{M}_{B_1} ^{k}(\mf{x'}, \mf{y'})) \\
  &  + m \cdot r(\op{M}_{f(A)} ^{k-1}(\mf{x'}, \mf{y'}) \xra{s_{a(t)+2}}  \op{M}_{B_2} ^{k}(\mf{x'}, \mf{y'})) \\
= & \sum \limits_{B' \in \overline{\cal{B'}}} m \cdot r(\op{M}_{A} ^{k}(\mf{x}, \mf{y}) \xra{s_{i(A,B')}}  \op{M}_{B'} ^{k+1}(\mf{x}, \mf{y})) \cdot e(\op{M}_{g(B')} ^{k}(\mf{x'}, \mf{y'})) \\
 &  + m \cdot r(\op{M}_{f(A)} ^{k-1}(\mf{x'}, \mf{y'}) \xra{t+1}  \op{M}_{B_1} ^{k}(\mf{x'}, \mf{y'})) \\
  &  + m \cdot r(\op{M}_{f(A)} ^{k-1}(\mf{x'}, \mf{y'}) \xra{t}  \op{M}_{B_2} ^{k}(\mf{x'}, \mf{y'})) \\
  = & \sum \limits_{B' \in \overline{\cal{B'}}} m \cdot r(\op{M}_{A} ^{k}(\mf{x}, \mf{y}) \xra{s_{i(A,B')}}  \op{M}_{B'} ^{k+1}(\mf{x}, \mf{y})) \cdot e(\op{M}_{g(B')} ^{k}(\mf{x'}, \mf{y'})) \\
&  + m \cdot r(\op{M}_{A} ^{k}(\mf{x}, \mf{y}) \xra{t+1}  \op{M}_{A} ^{k}(\mf{x}, \mf{y'})) \cdot e(\op{M}_{B_1} ^{k}(\mf{x'}, \mf{y'})) \\
  &  + m \cdot r(\op{M}_{A} ^{k}(\mf{x}, \mf{y}) \xra{t}  \op{M}_{A} ^{k}(\mf{x'}, \mf{y})) \cdot e(\op{M}_{B_2} ^{k}(\mf{x'}, \mf{y'}))
                        \end{align*}
\begin{align*}
= & d(m) \times r(\mf{x},\mf{y} \xra{t,t+1} \mf{x'},\mf{y'})\\
  & + m \times r(\mf{x},\mf{y} \xra{\es,t+1} \mf{x},\mf{y'}) \times r(\mf{x},\mf{y'} \xra{t,\es} \mf{x},\mf{y'}) \\
 &  + m \times r(\mf{x},\mf{y} \xra{t,\es} \mf{x'},\mf{y}) \times  r(\mf{x'},\mf{y} \xra{\es,t+1} \mf{x'},\mf{y'}) \\
=  & d(m) \times r(\mf{x},\mf{y} \xra{t,t+1} \mf{x'},\mf{y'})
   + m \times d(r(\mf{x},\mf{y} \xra{t,t+1} \mf{x'},\mf{y'})). \qed
\end{align*}

It is easy to see that the left $R$-module structure and the right $\drr$-module structure on $T$ are compatible:
$$a \cdot (m \times r)=(a \cdot m) \times r,$$
for $a \in R, r \in \drr$ and $m \in T$.
Hence we finally have the $q$-graded DG $(R, \drr)$-bimodule $T$.

\subsection{The functor $\mathcal{M} : DGP(\drr) \rightarrow DGP(R)$}
We show that tensoring with $T$ over $\drr$ maps the projective DG $\drr$-module $P(\mf{x}, \mf{y})=(\drr)e(\mf{x}, \mf{y})$ to a projective DG $R$-module in $DGP(R)$.

\begin{lemma} \label{tensor}
The tensor product $T \otimes_{\drr} P(\mf{x}, \mf{y})$ is the DG $R$-module $$T(\mf{x}, \mf{y})=\left(\bigoplus\limits_k T^k(\mf{x}, \mf{y}), \sum\limits_k d^k(\mf{x}, \mf{y})\right)$$ in $DGP(R)$ for any $\mf{x}, \mf{y} \in V(\g)$.
\end{lemma}

\proof
Since $T = \bigoplus\limits_{\mf{x'}, \mf{y'} \in V(\g)} T(\mf{x'}, \mf{y'})$ as left DG $R$-modules, where $T(\mf{x'}, \mf{y'}) \in DGP(R)$,
it follows that
$T \otimes P(\mf{x}, \mf{y})$ is the quotient of $\bigoplus\limits_{\mf{x'}, \mf{y'} \in V(\g)} (T(\mf{x'}, \mf{y'}) \times P(\mf{x}, \mf{y}))$ by the relations
$$\{(m \times r , e(\mf{x},\mf{y}))=(m , r \cdot e(\mf{x}, \mf{y}))~|~ m \in T(\mf{x'}, \mf{y'}), r \in \drr \}.$$
Since $T(\mf{x'}, \mf{y'}) \times P(\mf{x}, \mf{y})$ is spanned by $\{(m , r \cdot e(\mf{x}, \mf{y}))~|~ m \in T(\mf{x'}, \mf{y'}), r \cdot e(\mf{x}, \mf{y}) \neq 0\}$, $T \otimes P(\mf{x}, \mf{y})$ is spanned by
$$\{(m \times r , e(\mf{x},\mf{y}))~|~ m \in T(\mf{x'}, \mf{y'}), r \cdot e(\mf{x}, \mf{y}) \neq 0 \} \cong T(\mf{x}, \mf{y}).\qed$$

Since $DGP(\drr)$ is generated by the $P(\mf{x}, \mf{y})$'s, we obtain the functor $$\mathcal{M}: DGP(\drr) \xrightarrow{T \otimes_{\drr}-} DGP(R).$$
The following lemma implies that we have an induced functor $$\cal{M}|_{H^0} : H^0(DGP(\drr)) \rightarrow H^0(DGP(R)).$$

\begin{lemma}
The functor $\mathcal{M}$ preserves closed and exact morphisms.
\end{lemma}
\proof
For any $g \in \cal{H}om_{DGP(\drr)}(N, N')$, we have
$$\cal{M}(g) = id_T \otimes g  \in \cal{H}om_{DGP(R)}(T \otimes N, T\otimes N').$$
It suffices to prove $d(id_T \otimes g)=id_T \otimes d(g)$.
For any $t \in T, n \in N$,
\begin{align*}
(d(id_T \otimes g)) (t\otimes n) & = d\circ (id_T \otimes g) (t\otimes n) + (id_T \otimes g) \circ d(t\otimes n) \\
& = d(t \otimes g(n))+ (id_T \otimes g)(d(t)\otimes n+t\otimes d(n)) \\
& = d(t) \otimes g(n) + t \otimes d(g(n)) + d(t)\otimes g(n)+t\otimes g(d(n)) \\
& =  t \otimes d(g(n))+t\otimes g(d(n)) \\
& = (id_T \otimes d(g))(t \otimes n). \qed
\end{align*}

Note that $\cal{M}|_{H^0}$ is an exact functor since $\cal{M}$ also preserves mapping cones.
Then $\cal{M}|_{H^0}$ induces a $\Z [q^{\pm1}]$-linear map
$
K_0(\cal{M}|_{H^0}):  K_0({\rr}) \ra K_0(R)
$
under the isomorphisms
$$K_0(H^0(DGP(\drr))) \cong K_0({\rr}), \quad K_0(H^0(DGP(R))) \cong K_0(R).$$

\begin{proof} [Proof of Theorem \ref{Clifford}]
The algebra structure on $K_0(R)$ was proved in Proposition \ref{K_0}.
In order to prove that $\cal{M}_n|_{H^0}: H^0(DGP(\drnn)) \ra H^0(DGP(\rn))$ categorifies the multiplication $\op{m}_n$,
we compute $K_0(\cal{M}|_{H^0})$ using $\{[P(\mf{x}, \mf{y})]\}$ as a basis for $K_0(\rr)$. By Remark \ref{class} and Lemma \ref{tensor} we have
\begin{align*}
K_0(\cal{M}|_{H^0})(\mf{x}, \mf{y})= & K_0(\cal{M}|_{H^0})([P(\mf{x}, \mf{y})]) \\
= & [T \otimes P(\mf{x}, \mf{y})] \\
= & [T(\mf{x}, \mf{y})] \\
= & \sum \limits_{k}[T^k(\mf{x}, \mf{y})]h^k|_{h=-1} \\
= & \sum \limits_{k}\op{M}^k(\mf{x}, \mf{y})h^k|_{h=-1} \\
= & \op{m}(\mf{x}, \mf{y}).
\end{align*}
Hence, we finish the proof of Theorem \ref{Clifford}.
\end{proof}

\section{A categorification of $\mf{U}_n$ via a subcategory of $H^0(DGP(\rn))$}

\subsection{$\mf{U}_n$ as a subalgebra of $K_0(\rn)$}
We include $\mf{U}_n$ into $K_0(\rn)$ as a subalgebra for $n>0$.
\begin{lemma} There is an inclusion of $\Z[q^{\pm1}]$-algebras:
$$
\begin{array}{ccccc}
\imath_n : & \mf{U}_n & \ra & K_0(\rn) \\
& 1 & \mapsto & 1 \\
& E & \mapsto & \sum \limits_{\tiny{\begin{array}{c}0\leq i \leq n\\i \hspace{.1cm} \mbox{\em even}\end{array}}} X_i \\
& F & \mapsto & \sum \limits_{\tiny{\begin{array}{c}0\leq i \leq n\\i \hspace{.1cm} \mbox{\em odd}\end{array}}} X_i
\end{array}
$$
\end{lemma}
\proof
It suffices to show that $\imath_n$ maps the relations of $\mf{U}_n$ in Definition \ref{U_n} to the relations of $K_0(\rn)$
in Proposition \ref{K_0}:
$$\left(\imath(E)\right)^2=(\sum \limits_{i \hspace{.1cm} \mbox{even}} X_i)^2=\sum \limits_{i \hspace{.1cm} \mbox{even}} X_i^2+\sum \limits_{i<j \hspace{.1cm} \mbox{even}}(X_iX_j+X_jX_i)=0=\imath(E^2).$$
Similarly $(\imath(F))^2=0=\imath(F^2)$. We also have
\begin{equation*}
\begin{aligned}
\imath(E)\imath(F)+\imath(F)\imath(E) &= \sum \limits_{\tiny{\begin{array}{c}i \hspace{.1cm} \mbox{even}\\j \hspace{.1cm} \mbox{odd}\end{array}}} X_iX_j + \sum \limits_{\tiny{\begin{array}{c}i \hspace{.1cm} \mbox{even}\\j \hspace{.1cm} \mbox{odd}\end{array}}} X_jX_i \\
&=\sum \limits_{i =0}^{n-1} (X_iX_{i+1}+X_{i+1}X_i)+\sum \limits_{i<j-1}(X_iX_j+X_jX_i) \\
&=\sum \limits_{i =0}^{n-1} q^{2i+1-n}+0 \\
&= \imath(EF+FE). \qed
\end{aligned}
\end{equation*}

Similarly, we have an inclusion $\imath_{n,n}=\imath_n \otimes \imath_n: \mf{U}_n \otimes \mf{U}_n \ra K_0(\rnn)$.

Hence $\mf{U}_n$ and $\mf{U}_n \otimes \mf{U}_n$ can be viewed as subalgebras of $K_0(\rn)$ and $K_0(\rnn)$, respectively.
The restriction of the multiplication map $\op{m}: K_0(\rnn) \ra K_0(\rn)$ to $\mf{U}_n \otimes \mf{U}_n \ra \mf{U}_n$ gives the algebra structure on $\mf{U}_n$. We will lift subalgebras to subcategories in the next section.

\subsection{A subcategory of $H^0(DGP(\rn))$ categorifying $\mf{U}_n$}
Since $K_0(\rn)$ is isomorphic to the Grothendieck group of $H^0(DGP(\rn))$, we can formally construct $\mathcal{U}_n$ as a triangulated full subcategory of $H^0(DGP(\rn))$ whose Grothendieck group is the subalgebra $\mf{U}_n$.
We define a bifunctor
$$\chi_n : H^0(DGP(\rn)) \times H^0(DGP(\rn)) \ra H^0(DGP(\rnn)) \ra H^0(DGP(\drnn)),$$
where the first map is given by tensoring two DG $\rn$-modules over $\F$ and the second map is an inverse of the equivalence in Lemma \ref{equivalence} which maps $P'(\mf{x},\mf{y})$ to $P(\mf{x},\mf{y})$ for any pair $\mf{x},\mf{y} \in \gn$.
Let $\rho_n=\mathcal{M}_n|_{H^0} \circ \chi_n:$
$$H^0(DGP(\rn)) \times H^0(DGP(\rn)) \ra H^0(DGP(\drnn)) \ra H^0(DGP(\rn)).$$
Notice that $\rho_n (M, P([\es])))=\rho_n (P([\es]), M)=M$, for any $M \in H^0(DGP(\rn))$.

To define $\mathcal{U}_n$, we first lift
$1$ to $P([\es])$,
$q$ to $P([\es])\{1\}$, and
$q^{-1}$ to $P([\es])\{-1\}$.
Letters $E$ and $F$ are lifted to
$$\mathcal{E}=\bigoplus_{i \hspace{.1cm} \mbox{even}} P([i]) \in H^0(DGP(\rn)),$$
$$\mathcal{F}=\bigoplus_{i \hspace{.1cm} \mbox{odd}} P([i]) \in H^0(DGP(\rn)).$$
Then for the multiplication $A_1A_2$ of $A_1, A_2 \in \{1, q, q^{-1}, E, F\}$, we lift it to
$$\mathcal{A}_1\mathcal{A}_2=\rho_n (\mathcal{A}_1, \mathcal{A}_2),$$
where $\mathcal{A}_i$ is the lifting of $A_i$ defined above for $i=1, 2$.
For multiplication of $3$ letters $A_1, A_2, A_3 \in \{1, q, q^{-1}, E, F\}$, we have different lifting of multiplication for different orders.
For instance, we lift $(A_1A_2)A_3$ to
$$\rho_n(\rho_n (\mathcal{A}_1, \mathcal{A}_2), \mathcal{A}_3),$$
and $A_1(A_2A_3)$ to
$$\rho_n(\mathcal{A}_1, \rho_n(\mathcal{A}_2, \mathcal{A}_3)).$$
For multiplication of more letters, the definition of lifting is similar.

Then we define $\mathcal{U}_n$ as the smallest triangulated full subcategory of $H^0(DGP(\rn))$ containing the lifting of multiplication of all finitely many letters in $\{1, q, q^{-1}, E, F\}$ for all possible orders.

\begin{rmk} \label{E2=0}
An equation in $\mf{U}_n$ may not be lifted to an isomorphism in $\mathcal{U}_n$.
For example, the equation $E^2 = 0 \in \mf{U}_n$ is lifted to
\begin{equation*}
\begin{aligned}
\mathcal{E}\mathcal{E}=\bigoplus_{i,j \hspace{.1cm} \mbox{even}} \mathcal{M}_n(P([i], [j])) \in H^0(DGP(\rn)).
\end{aligned}
\end{equation*}
As a cochain complex, $\mathcal{E}\mathcal{E}=(\mathcal{E}\mathcal{E})^{-1}\oplus(\mathcal{E}\mathcal{E})^{0}$ has zero differential, where
$$(\mathcal{E}\mathcal{E})^{-1}=(\mathcal{E}\mathcal{E})^{0}=\bigoplus_{\tiny{\begin{array}{c}i,j \hspace{.1cm} \mbox{even}\\i>j\end{array}}} P([i, j]).$$
It is not isomorphic to $0 \in H^0(DGP(\rn))$.
\end{rmk}

Next we define $\mathcal{U}_{n,n}$ as the smallest triangulated full subcategory of $H^0(DGP(\drnn))$ containing $\{\chi_n(\mathcal{X}, \mathcal{Y})~|~ \mathcal{X}, \mathcal{Y} \in \mathcal{U}_n\}$.

\begin{proof} [Proof of Theorem \ref{main}]
Since $\mf{U}_n$ is generated by $q, q^{-1}, E$ and $F$ as an algebra and $\mathcal{U}_n$ is the smallest triangulated full subcategory containing the lifting of multiplication of $q, q^{-1}, E$ and $F$, it follows that $K_0(\mathcal{U}_n) = \mf{U}_n$.
Similarly, we have $K_0(\mathcal{U}_{n,n})=\mf{U}_n \otimes \mf{U}_n$.

Since the exact functor $\mathcal{M}_n|_{H^0} : H^0(DGP(\drnn)) \rightarrow H^0(DGP(\rn))$ maps $\mathcal{U}_{n,n}$ into $\mathcal{U}_{n}$, let $\mathcal{F}_n : \mathcal{U}_{n,n} \rightarrow \mathcal{U}_{n}$ be the restriction.
Then $K_0(\mathcal{F}_n): K_0(\mathcal{U}_{n,n}) \ra K_0(\mathcal{U}_n)$ agrees with the multiplication $\op{m}_n:\mf{U}_n \otimes \mf{U}_n \ra \mf{U}_n$.
Hence we proved Theorem \ref{main}.
\end{proof}

\end{document}